\documentclass[12pt]{amsart}
\usepackage[colorlinks=true,citecolor=green,linkcolor=magenta]{hyperref}
\usepackage{amsmath}
\usepackage{verbatim}
\usepackage{amsfonts}
\usepackage{amssymb}
\usepackage{blkarray}
\usepackage{color,colortbl}
\usepackage[all]{xy}  
\usepackage{enumerate}
\usepackage[top=1in, bottom=1in, left=1in, right=1in]{geometry}
\usepackage{mathrsfs}

\usepackage{stmaryrd}
\usepackage{bbold}
\renewcommand{\k}{\mathbb{k}}
\usepackage{cleveref}
\usepackage{soul}
\usepackage{arydshln,xcolor}
\usepackage{tikz-cd}
\usepackage{graphicx}
\usepackage{comment}
\theoremstyle{definition}
\newtheorem{theorem}{Theorem}[section]
\newtheorem{theoremx}{Theorem}
\numberwithin{equation}{section}

\newtheorem{question}[theorem]{Question}

\newtheorem{corollary}[theorem]{Corollary}

\newtheorem{proposition}[theorem]{Proposition}

\theoremstyle{definition}
\newtheorem{definition}[theorem]{Definition}

\newtheorem{example}[theorem]{Example}

\newtheorem{remark}[theorem]{Remark}

\newtheoremstyle{TheoremNum}
{8pt}{8pt}              %%% space between body and theorem
{\upshape}                      %%% theorem body font
{}                              %%% Indent amount (empty = no indent)
{\bfseries}                     %%% theorem head font
{.}                             %%% Punctuation after theorem head
{.5em}                             %%% Space after theorem head
{\theoremname{#1}\theoremnote{ \bfseries #3}}%%% theorem head spec
\theoremstyle{TheoremNum}

\newcommand{\m}{\mathfrak{m}}
\newcommand{\n}{\mathfrak{n}}

\newcommand{\cC}{\mathfrak{C}}

\newcommand{\edim}{\operatorname{edim}}
\newcommand{\Rank}{\operatorname{rank}}
\newcommand{\type}{\operatorname{type}}

\newcommand{\Hom}{\operatorname{Hom}}
\newcommand{\End}{\operatorname{End}}
\newcommand{\Ext}{\operatorname{Ext}}

\newcommand{\C}{\mathfrak{C}}
\newcommand{\Frac}{\operatorname{Frac}}

\newcommand{\quot}{\operatorname{Quot}}

\newcommand{\depth}{\operatorname{depth}}

\DeclareMathOperator{\PF}{PF}

\renewcommand{\leq}{\leqslant}
\renewcommand{\geq}{\geqslant}
\newcommand{\ds}{\displaystyle}

\newcommand{\ps}[1]{\llbracket {#1} \rrbracket}

\def\cm{\operatorname{CM}}
\def\rf{\operatorname{Ref}}

\title{Extremal behavior of reduced type of one dimensional rings}

\address{Department of Mathematics, University of Utah, Salt Lake City, UT, USA}

\author[Maitra]{Sarasij Maitra}
\email{maitra@math.utah.edu}
\address{Department of Mathematics, Indian Institute of Technology Delhi, Hauz Khas, India.}

\author[Mukundan]{Vivek Mukundan}
\email{vmukunda@iitd.ac.in}


\begin{document}

	\begin{abstract}
Let $R$ be a domain that is a complete local $\k$ algebra in dimension one. In an effort to address the Berger's conjecture, a crucial invariant reduced type $s(R)$ was introduced in \cite{huneke2021torsion}. In this article, we study this invariant and its max/min values separately and relate it to the valuation semigroup of $R$. We justify the need to study $s(R)$ in the context of numerical semigroup rings and consequently investigate the occurrence of the extreme values of $s(R)$ for the Gorenstein, almost Gorenstein, and far-flung Gorenstein complete numerical semigroup rings.  Finally, we  study the finiteness of the category $\cm(R)$ of maximal Cohen Macaulay modules and the category $\rf(R)$ of reflexive modules for rings which are of maximal/minimal reduced type and provide many classifications.
%Further, we discuss classes of rings where the conjecture can be quickly verified by studying the order of the units and the valuation semi-group. 
	\end{abstract}
	\maketitle
\section{Introduction}
Let $(R,\mathfrak{m},\k)$ be a non-regular one dimensional complete local domain which is a $\k$-algebra. We assume that $\k$ is algebraically closed of characteristic $0$. Assume $\k\ps{\alpha_1t^{a_1},\dots\alpha_nt^{a_n}}\cong R=\frac{\k\ps{X_1,\dots, X_n}}{I}$, $I\subseteq (X_1,\dots,X_n)^2$. Here $\alpha_i$ are units of $\overline{R}$. Additionally, we also arrange the $x_i$ such that $a_1<a_2<\dots<a_n$. The \textit{reduced type} of $R$ is
\begin{align*}
    s(R):=\dim_\k\frac{\cC+x_1R}{x_1R},
\end{align*}
where $\C$ is the conductor ideal of the ring $R$. Recall that the conductor $\C=R:_{\quot(R)}\overline{R}$ is the largest common ideal between $R$ and $\overline{R}$. This invariant was defined in \cite{huneke2021torsion} as a means of constructing an over ring $S$ of $R$. The authors also demonstrated that constructing sufficient torsion elements in the module of differentials $\Omega_S$ helped to solve the Berger's conjecture for $R$. This invariant was also used by the authors in \cite{MaitraMukundan23} extensively to build classes of rings which attests to the positivity of the conjecture.   The main objective of this article is to investigate the values attained by this invariant $s(R)$ and its relationship to the Cohen-Macaulay type $\type(R)$. As the name suggests, the reduced type $s(R)$ is indeed related to $\type(R)$. For, $\frac{\C+x_1R}{x_1R}\subseteq 
\frac{x_1R:\m}{x_1R}$ and the $\k$-vector space dimension of the latter object is precisely $\type(R)$. Thus the maximum and minimum values of $s(R)$ are captured in the inequality
%\begin{align*}
 $$   1\leq s(R)\leq \type(R).$$
%\end{align*}
Following this observation, we say that $R$ is of \textit{minimal} (resp. \textit{maximal}) reduced type when $s(R)=1$ (resp. $s(R)=\type(R)$). Interestingly, when $R$ is Gorenstein (equivalently $\type(R)=1$), $s(R)$ equals one, but not conversely (see for example \Cref{reduced type one but not gorenstein}). In fact, Gorenstein rings are the only ones of both maximal and minimal reduced type.

Notice that under the above setup, the integral closure $\overline{R}=\k\ps{t}$ is a $DVR$ (Discrete Valuation Ring) and hence one can define an order valuation $v(\cdot)$ on the ring $R$ induced from the one on $\overline{R}$, allowing us to define a valuation semigroup $v(R)$ of $R$. There is a strong relationship between the singularities of $R$ and the structural representations of the valuation semigroup $v(R)$. For example, Kunz \cite{Kunz1970} showed that $R$ is Gorenstein if and only if $v(R)$ is symmetric. This has since then been generalized to almost Gorenstein rings too. Barucci-Froberg \cite{Barucci-Froberg97} showed that $R$ is almost Gorenstein (equivalently $\ell(\overline{R}/R)=\ell(R/\C)+\type(R)-1$, where $\ell(M)$ is the length of the module $M$) if and only if the valuation semigroup $v(R)$ is  almost symmetric and $\type(R)=\type(v(R))$. The Cohen-Macaulay type $\type(v(R))$ can be easily computed as the cardinality of the pseudo-Frobenius set $\PF(v(R))$ (see \Cref{Section on numerical semigroups} for a definition). In fact there are some nice characterizations to check when the valuation semigroup $v(R)$ is almost symmetric (\cite{nari2013symmetries}). We show that the reduced type $s(R)$ is also the cardinality of the set $[c-a_1,c-1]\backslash v(R)$ where $\C=(t^c)\overline{R}$, the conductor of the ring $R$ (see \Cref{redtypeinterpret}). The integer $c$ is called the conductor valuation of $R$. For the ring $R$ we can also define a complete numerical semigroup ring $\k\ps{v(R)}=\k\ps{t^{b_1},\dots,t^{b_m}}$ where $v(R)=\langle b_1,\dots,b_m\rangle$. Along with \cite[Proposition II.1.16]{Barucci-Froberg-memoirs} the following set of inequalities become apparent.  $$s\left(\k\ps{v(R)}\right)=s(R)\leq \type(R)\leq |\PF(v(R))|=\type(\k\ps{v(R)}).$$
Using this set of inequalities we conclude:
\begin{theoremx}[\Cref{transitiontosemigroups}]Let $R$ be as described at the beginning of the article.
    \begin{enumerate}
\item The ring $R$ is of minimal reduced type if and only if $\k\ps{v(R)}$ is of minimal reduced type.
   \item The ring $R$ is of maximal reduced type if $\k\ps{v(R)}$ is of maximal reduced type. The converse holds if additionally $\type(R)=\type(\k\ps{v(R)})$.
   \end{enumerate}
   In fact, when $R$ is of minimal multiplicity or when $R$ is almost Gorenstein, then the statements in (2) are equivalent.
\end{theoremx}
%Thus to study the reduced type $s(R)$ of $R$, it may be very fruitful to study the reduced type of the numerical semigroup $s(\k\ps{v(R)})$. 
The above result indicates that the exact nature of the minimal/maximal reduced type of $R$ can often be captured by studying the minimal/maximal reduced type of the complete numerical semigroup $\k\ps{v(R)}$. As such, we focus our attention to complete numerical semi-group rings and provide explicit characterization (\Cref{mainpropo} and \Cref{minredtypemainprop}) of maximal/minimal reduced type rings in terms of the distribution of the pseudo-Frobenius numbers. For instance, the above study yields the following criteria. 
%A more interesting behavior of complete numerical semigroup rings exihibiting maximal or minimal reduced type in terms of distribution of the valuations $a_1,\dots,a_n$ is codified in :

\begin{theoremx}[\Cref{minmultcrit,minmultminredtype}]
    Let $R=\k\ps{t^{a_1},t^{a_2},\ldots,t^{a_n}}$ be a complete numerical semigroup ring of minimal multiplicity. Then $R$ is of maximal (resp. minimal) reduced type if and only if $a_2\geq a_n-a_1+1$ (resp. $a_{n-1}+a_1-1<a_n$).
\end{theoremx}
When $R$ is almost Gorenstein, the maximal and minimal reduced type property bounds the conductor valuation $c$ of $R$ (\Cref{AlmostGorensteinVal}). But when the type is 2, these statements are actually equivalent. These types of rings are called \textit{Kunz rings}.
\begin{theoremx}[\Cref{pseudosymm}]Let $R=\k\ps{t^{a_1},t^{a_2},\ldots,t^{a_n}}$ be a complete numerical semigroup ring which is almost Gorenstein of type 2 with conductor valuation $c$. Then $R$ is of maximal (resp. minimal) reduced type if and only if $c\leq 2a_1-1$ (resp. $c>2a_1-1$).
\end{theoremx}

Herzog, Kumashiro and Stamate in \cite{herzog2021tiny} introduced the notion of far-flung Gorenstein rings (see \Cref{section of far flung gorenstein} for a definition). We show that the reduced type of such rings are at least $2$ (\Cref{redtype2}) and hence they are never of minimal reduced type. But the maximal reduced type can be understood from constraints in $\type(R)$.
\begin{theoremx}[\Cref{ffg type leq 3,ffg type 4 non min mult}] Let $R$ be a far-flung Gorenstein complete numerical semigroup ring of $\type(R)\leq 3$ or $\type (R)=4 $ of non-minimal multiplicity. Then  $R$ has maximal reduced type.
%        \item Suppose $R$ is far-flung Gorenstein with $a_1>\overline{n}(\type(R)-1)+1$ then $R$ is of maximal reduced type where $\overline{n}(r)$ refers to the Rohrbach number of the integer $r$.
\end{theoremx}
Our methods do not immediately generalize to higher types in such rings. The lack of answers to \cite[Questions 5.5]{herzog2021tiny} is the primary obstruction. 
%becomes bigger it is important to know the generators of $v(R)$  before making any conclusions on the maximal reduced type of such rings. This is also one of the important questions posed in \cite{herzog2021tiny}. 
Bounds on multiplicity of $R$ also have implications on the maximal reduced type of such rings (\Cref{ffG rohrbach maxreduced type} and \Cref{ffG rohrback type 5}). 

In \Cref{section on gluing} we study the gluing on two semigroups as one of the interesting means of constructing more minimal reduced type examples.

In \Cref{categories}, we study the finiteness of the categories of maximal Cohen-Macaulay modules, $\cm(R)$, and the category of reflexive modules, $\rf(R)$, primarily in the context of numerical semigroup rings of maximal/minimal reduced type (refer to \Cref{categories} for the relevant definitions). A category is said to be of finite type if there are only finitely many indecomposable objects upto isomorphism. Notably, such finiteness studies have been of interest in both algebra and representation theory; the central idea of the latter is to analyze an algebraic object based the way it acts on sets. The representation theory of maximal Cohen-Macaulay modules has garnered high interest in the last 50 years, for instance these modules play an important role in Hochster's study of homological conjectures. We refer the reader to various sources such as \cite{herzog1978ringe}, \cite{knorrer1986cohen},\cite{Greuel-Knorrer85},  \cite{buchweitz1987cohen}, \cite{solberg1989hypersurface}, \cite{Yoshino90}, \cite{wiegand1998local}, \cite{leuschke2012cohen}, etc.  which discuss such classifications in great details. In particular, it is well-known that for finiteness of $\cm(R)$, one only needs to focus on multiplicity at most $3$. In \Cref{fintyperes1}, we first establish the explicit forms of numerical semigroup rings of minimal multiplicity $3$ such that $\cm(R)$ is of finite type and we use this as our tool to study $\rf(R)$ subsequently. The study of the category $\rf(R)$ has attracted quite a bit of attention recently, for instance see \cite{kobayashi2022set}, \cite{kobayashi2022syzygies}, \cite{dao2023reflexive}, etc. The categories $\rf(R)$  and $\cm(R)$ coincide when $R$ is Gorenstein due to the work of Vasconcelos \cite{vasconcelos68}. Outside the Gorenstein situation, however, the finiteness problem still is open. We provide explicit descriptions of the almost Gorenstein scenario in the context of maximal/minimal reduced type.

%{\color{red}ref type intro and its implications here}

%{\color{red} Use this result in intro explanation : If $R=\k\ps{H}$ is of minimal reduced type, then $\cm(R)$ is of finite type if and only if $R$ is complete $ADE$ hypersurface singularity of multiplicity $2$ or $3$.}
\begin{theoremx}[\Cref{almostGorensteinminredtyperef,ref finite e leq 2,ref finite e =3}, \Cref{minmult3almostgorensteinref}]Suppose $R$ is a complete numerical semigroup ring such that $R$ is almost Gorenstein  but not Gorenstein. Then the following statements are true.
    \begin{enumerate}
        \item Suppose $R$ is of minimal reduced type. Then $\rf (R)$ is of finite type if and only if $R$ is either $\k\ps{t^3,t^7,t^{11}}$ or $\k\ps{t^3,t^8,t^{13}}$. 
        \item Suppose $R$ is of maximal reduced type 2. If $R$ is of minimal multiplicity, then $\cm(R)$ (and hence $\rf(R)$) is of finite type).
        \item Suppose $R$ is of maximal reduced type and $e(\m:\m)\leq 2$. Then $\rf(R)$ is of finite type if and only if $R$ is either $\k\ps{t^{a_1},t^{a_1+1},\dots,t^{2a_1-1}}$ or $\k\ps{t^{a_1},t^{a_1+2},t^{a_1+3},\dots,t^{2a_1+1}}$.
        \item Suppose $R$ is of maximal reduced type and $e(\m:\m)=3$, then $\rf(R)$ is of finite type if and only if $R=\k\ps{t^{a_1},t^{a_1+1},t^{a_1+3},\dots,t^{2a_1-1}}, a_1\geq 4$.
    \end{enumerate}
\end{theoremx}
\subsection*{Acknowledgements}  Much of this article is dedicated to answering Craig Huneke's questions about reduced type. We owe him a great deal for the understanding of this invariant. Additionally, we would like to thank Anurag Singh, Kei-ichi Watanabe, and Hailong Dao for various discussions and for providing us with valuable references.

%\subsection{Convention}$S=\k\ps{H}$

%\section{Preliminaries?}
\section{Maximal and Minimal Reduced type for one dimensional complete local domains}
Let $(R,\mathfrak{m},\k)$ be a non-regular one dimensional complete local domain which is a $k$-algebra. We assume that $\k$ is algebraically closed of characteristic $0$. The integral closure of $R$ in its field of fractions $K$ if denoted by $\overline{R}$. Notice that $\overline{R}=\k\ps{t}$. Using this and the fact that $R=\frac{\k\ps{X_1,\dots, X_n}}{I}$, $I\subseteq (X_1,\dots,X_n)^2$, we can write each of the images of $X_i$ modulo $I$, denoted $x_i$, as $x_i=\alpha_it^{a_i}$. Here $\alpha_i$ are units of $\overline{R}$. We also arrange the $x_i$ such that $a_1<a_2<\dots<a_n$. Here $n$ is the embedding dimension of $R$, which is defined as the minimal number of generators of the maximal ideal $\m$, denoted $\edim(R)$. By our assumption, $n\geq 2$.

There is a valuation on $\overline{R}$ (a DVR) which in turn induces a valuation on $R$ given by 
\begin{align*}
    v\left(\sum_i c_it^{i}\right)=\min \{i~|~c_i\neq 0\}\text{ where }c_i\in\k.
\end{align*}
Also,  we set $v(\alpha)=0$ if $\alpha$ is a unit in $\overline{R}$. This valuation extends to  the valuation on $\Frac R$, the fraction field of $R$, given by $v(\frac{\alpha}{\beta})=v(\alpha)-v(\beta)$ for $\alpha,\beta\in R$. For any $R$ submodule $M$ of $\Frac R$, we let $v(M)$ to be the collection $\{ v(\alpha)~|~\alpha\in M\}$. The set $v(R)$ forms a semigroup called the valuation semigroup of $R$. 

%A complete numerical semigroup ring is of the from $\k\ps{t^{b_1},\dots,t^{b_m}}$ and its valuation semigroup is generated by $b_1,\dots,b_m$. 

%Given a ring $R$ as above, we can generate a complete numerical semigroup ring $\k\ps{v(R)}$. There is a relationship between $R$ and $\k\ps{v(R)}$ which we defer until the next section.

Let $\cC=R:_K\overline{R}$ denote the conductor ideal of $R$ in $\overline{R}$. Observe that $\cC=(t^{c+i})_{i\geq 0}$, where $c$ is the positive integer such that $t^{c-1}\not \in R$ but $t^{c+i}\in R$ for all $i\geq 0$. 
\begin{remark}\label{multiplicityandconductor}
Notice that $\m\overline{R}=x_1\overline{R}$ and thus $$\m=\m\overline{R}\cap R=x_1\overline{R}\cap R=\overline{x_1},$$ i.e., $x_1$ is a minimal reduction of $\m$.  Since $\m\overline{R}=x_1\overline{R}$, we get that $\m\cC=x_1\cC$. Hence, $\mu(\cC)=\ell(\cC/x_1\cC)=e(x_1;\cC)=e(x_1;R)\Rank_R(\cC)=\ell(R/x_1R)\Rank_R(\cC)=e(R)$. The last few equalities follow from well-known properties of the Hilbert-Samuel multiplicity, and also the fact that a non-zero ideal is a maximal Cohen-Macaulay module of rank one. 
\end{remark}
The following result which we will keep using in the sequel is due to Herzog and Kunz. Many of the length computations reduce to counting arguments in the valuation semigroup.

\begin{proposition}{\cite[Proposition 2.9]{herzog1971wertehalbgruppe}}\label{herzog-kunz-length}
    Suppose $M_1\subseteq M_2$ are two $R$ submodules of $\Frac R$ with $M_1\neq 0$. Then $M_2/M_1$ has finite length and its length $\ell(M_2/M_1)$ is the number of elements in $v(M_2)$ but not in $v(M_1)$.
\end{proposition}

\begin{definition}\cite{huneke2021torsion}
We define the reduced type of $R$, denoted $s(R)$ as follows:
$$s(R):=\dim_\k\frac{\cC+x_1R}{x_1R}$$
\end{definition}

Since $\m\cC=x_1\cC\subseteq x_1R$, the above definition makes sense. Also notice by the same inclusions, we get that $s(R)\leq \dim_\k\frac{x_1R:_R\m}{x_1R}$. The latter quantity is precisely the type of $R$ and this justifies the nomenclature. Further, it is well-known that $\cC$ is not contained in any proper principal ideal. Thus, we get that $$1\leq s(R)\leq \type(R).$$

\begin{definition}
We say $R$ has maximal reduced type if $s(R)=\type(R)$. Analogously, we say that $R$ has minimal reduced type if $s(R)=1$.
\end{definition}

\begin{question}\label{question1}
Can we classify when $R$ has maximal or minimal reduced type? 
\end{question}

Notice that if $R$ is Gorenstein, then $R$ has maximal (and minimal) reduced type since $\type(R)=1$. However, $s(R)=1$, does not force $R$ to have type one. 
\begin{example}\label{reduced type one but not gorenstein}
Let $R=\k\ps{t^5,t^8,t^9,t^{11}}$. The conductor $\C=(t^{13})\overline{R}$ and the valuation semigroup is $\{5,8,9,10,11,13\rightarrow\}$. Thus the reduced type $s(R)=1$, but the $\type(R)=2$ and hence not a Gorenstein ring.
\end{example}
\begin{definition}We say that $R$ is of \textit{minimal multiplicity} if $\edim R=e(R)$. \end{definition}
In the rest of the article $\ell(M)$ denotes the length of a module $M$.
\begin{proposition}\label{lengthconditions}
Let $(R,\m,\k)$ be as above. Let $\cC$ denote the conductor and $x_1$ denote the minimal reduction of $\m$ as in the definition of $s(R)$. Then $R$ has maximal reduced type if any of the following holds. \begin{enumerate}
    \item $\ell(R/\cC)\leq 2$.
    
    \item $\ell(R/\cC)=3$ and $R$ does not have minimal multiplicity.
\end{enumerate}
\end{proposition}

\begin{proof}
Notice that maximal reduced type implies that we have the equality $$\cC+x_1R=x_1R :_R \m.$$ 
Observe that we always have the following inclusions: $$\cC\subseteq \cC+x_1R\subseteq x_1R :_R \m\subseteq \m.$$ 
If $\cC+x_1R=\cC$, then $x_1\in \cC$. 

If $\ell(R/\cC)=1$, then $\cC=\m$ and hence all the inclusions are equalities and we are done. So assume $\ell(R/\cC)=2$. If $\cC+x_1R=\cC$, then $x_1\in \cC$. Thus, $\m=\overline{x_1}\subseteq \overline{\cC}=\cC\subseteq \m$. This shows that $\m=\cC$ but this contradicts the fact, that $\ell(R/\cC)=2$. Hence, $\cC\subsetneq \cC+x_1R$. Since $\ell(\m/\cC)=1$, this shows that $\cC+x_1R=x_1:_R \m = \m$ and we are done with the proof of $(a)$. 

For $(b)$, first we recall the well-known fact that $R$ is of minimal multiplicity if and only if $\m^2=x_1\m$. We include a quick proof here. Since $\edim(R)=e(R)$, we have that $\ell(\m/\m^2)=\ell(R/x_1R)$. The LHS is $\ell(R/\m^2)-1$ whereas the RHS is $\ell(R/x_1\m)-1$ (using the inclusions $\m^2\subseteq \m\subseteq R$ and $x_1\m\subseteq x_1R\subseteq R$ respectively). Thus, $\ell(R/x_1\m)=\ell(R/\m^2)$ which shows that $\m^2=x_1\m$. The converse is clear as well as the arguments are reversible. 

Since $x_1R$ is a minimal reduction of $\m$, we get that $x_1R\cap \m^2=x_1\m$. Hence, $\m^2=x_1\m$ if and only if $x_1R:_R \m=\m$. Since by assumption, $R$ is not of minimal multiplicity, we get the following chain of ideals: $$\cC\subsetneq \cC+x_1R\subseteq x_1R:_R\m\subsetneq \m.$$ The conclusion now follows because $\ell(\m/\cC)=2$. 
\end{proof}

\begin{remark}
The arguments in the first part show that $\m=\cC$ if and only if $\cC=\cC+x_1R$. In fact, one can precisely describe $R$ in the case $\ell(R/\cC)=2$ (see \cite[Proposition 6.3 (2)]{dao2023reflexive}).
\end{remark}

The following example shows that the minimal multiplicity condition cannot be ignored. The computation of $\type(R)$ in the following examples can be obtained by computing the cardinality of the pseudo-Frobenius set $\PF(v(R))$ (we refer to the next section for a definition of pseudo-Frobenius set).

\begin{example}
Let $R=\k\ps{t^4,t^6,t^9,t^{11}}$. Here $\cC=(t^8,t^9,t^{10},t^{11})$. Thus, $\ell(R/\cC)=3$ (using \Cref{herzog-kunz-length}, it is enough to notice that the valuations missing from the conductor are $0,4,6$). Also, $R$ has minimal multiplicity. Finally note that $s(R)=2$ whereas $\type(R)=3$. 
\end{example}

For higher colength, the equality can fail as the following examples show.
\begin{example}
Let $R=\k\ps{t^4,t^6,t^{11},t^{13}}$. Here, $\cC=(t^{10},t^{11},t^{12},t^{13})$ and thus $\ell(R/\cC)=4$. Also, $R$ has minimal multiplicity. But $s(R)=2$ whereas $\type(R)=3$. 
\end{example}

%\begin{question}
%What about $\ell(R/\cC)=4$ and non-minimal multiplicity? 
%\end{question}

\begin{example}
Let $R=\k\ps{t^5,t^7,t^8,t^{11}}$. The conductor is $\cC=(t^{10})\overline{R}$. We get that $\ell(R/\cC)=4$, $s(R)=2, \type(R)=3$. Here $R$ is not of minimal multiplicity.
%$\operatorname{PF}(H)=\{3,6,9\}$.
\end{example}
\begin{theorem}\label{redtypeinterpret}
The quantity $s(R)$ equals the number of integers in the interval $[c-a_1,c-1]$ that are missing from the value group of $R$. 
\end{theorem}

\begin{proof}
Notice that $s(R)=\ell(\cC/(\cC\cap x_1R))$. Hence, we need to count the valuations that are in $\cC$ but missing from the valuation set coming from $\cC\cap x_1R$. Since $\m\cC=x_1\cC\subseteq \cC\cap x_1R$, the only relevant valuations must come from $\cC\setminus \m\cC$. Hence, we can focus on the following generating set of $\cC$: $t^c,\dots, t^{c+a_1-1}$. 

Suppose for some $1\leq i\leq a_1$, $t^{c+a_1-i}=x_1\alpha$ for some $\alpha\in R$. Then we get that $c+a_1-i=a_1+k$ for some $k$ lying in the value group of $R$. Hence, $k=c-i\in [c-a_1,c]$. Thus the missing valuations from $[c-a_1,c-1]$ precisely correspond to the elements in the generating set which cannot be written as a multiple of $x_1$.
\end{proof}

\begin{corollary}\label{redtypestaysunchanged}
We have $s(R)=s(\k\ps{v(R)})$. 
\end{corollary}

\begin{proof}
    The proof is clear from \Cref{redtypeinterpret}.
\end{proof}

%Before proceeding, we provide an interpretation of $s(R)$ purely in terms of valuations. 

\section{Numerical Semigroup Rings}\label{Section on numerical semigroups}
%Associated with $R$ is the valuation semigroup $v(R)$. We can create the complete numerical semigroup ring $S=\k\ps{v(R)}$ corresponding to $v(R)$. 
For a semigroup $H$ minimally generated by $b_1<\cdots<b_m$, we define a complete numerical semigroup $\k\ps{H}=\k\ps{t^{b_1},\dots,t^{b_m}}$. So for the numerical semigroup $v(R)$ appearing in the previous section, we can generate a complete numerical semigroup ring $\k\ps{v(R)}$.

It is a natural question whether our study can be transferred to the study of this numerical semigroup ring. For that we need the interpretations of type and reduced type in terms of valuations. One of the simplest relationships between $R$ and $\k\ps{v(R)}$ is that  their conductor valuations remain the same.

It is well-known that the type of a complete numerical semigroup ring  $\k\ps{H}$ is given by the cardinality of the pseudo-Frobenius set, denoted $\operatorname{PF}(H)$. It is defined as follows:
$$\operatorname{PF}(H)=\{x\in \mathbb{Z}\setminus H\mid x+h\in H \text{~for all~}0\neq h\in H\}.$$ The largest element in this set is called the Frobenius number of $H$, denoted $F(H)$. In our notations, $F(H)=c-1$.

In the previous section we saw that the reduced type is invariant under the study of $R$ versus $\k\ps{v(R)}$. But the type of $R$ is more sensitive through this change of rings. The next result \Cref{typecomparison} which appears in \cite[Proposition II.1.16]{Barucci-Froberg-memoirs}, captures this behaviour. We provide a proof for the convenience of the reader. Also, in the following $|S|$ refers to the cardinality of a set $S$.
\begin{remark}\label{remark on HomI,J}
    We recall that for any two ideals of $R$, we can identify $\Hom_R(I,J)$ with $J:I=\{\alpha\in \operatorname{Frac}(R)\mid \alpha I\subseteq J\}$ via the following: for any non-zero $x,y\in I$ and $f:I\to J$, we have $\frac{f(x)}{x}=\frac{f(y)}{y}$ using $R$-linearity of $f$. Hence, for $0\neq r\in I$, we get $f(r)=r\frac{f(rx)}{rx}=r\frac{f(x)}{x}$, a multiplication by $\frac{f(x)}{x}\in \operatorname{Frac}(R)$.
\end{remark}

\begin{theorem}\label{typecomparison}
We have the inequality $$s(R)\leq \type(R)\leq |\PF(v(R))|=\type(\k\ps{v(R)}).$$
\end{theorem}

\begin{proof}
    The first inequality follows from our discussion of reduced type above. For the second one, we make the following observations. Let $\omega_R$ be a canonical ideal of $R$. Then $\type(R)=\ell(\omega_R/\m\omega_R)=\ell(\Hom(\omega_R/\m\omega_R,E(\k)))$,  where $E(\k)$ is the injective hull of $\k$. The latter equality is due to \cite[Proposition 3.2.12(b)]{bruns_herzog_1998} while the former is due to \cite[Theorem 3.3.11(c)]{bruns_herzog_1998}. Applying $\Hom_R(\cdot, \omega_R)$ to the exact sequence $$0\to \m\omega_R\to \omega_R\to \omega_R/\m\omega_R\to 0,$$ we obtain $$0\to \Hom_R(\omega_R,\omega_R)\to \Hom_R(\m\omega_R, \omega_R)\to \Ext^1_R(\omega_R/\m\omega_R,\omega_R)\to 0$$ where we used the facts that $\Hom_R(M,N)=0$ if $M$ is torsion and $N$ is torsion-free, and also that $\Ext^1_R(\omega_R,\omega_R)=0$ (\cite[Theorem 3.3.10(c)(iii)]{bruns_herzog_1998}) Now using local duality \cite[Theorem 3.5.8]{bruns_herzog_1998}, we get that $\Ext^1_R(\omega_R/\m\omega_R, \omega_R)\cong \Hom_R(H^0_\m(\omega_R/\m\omega_R),E(\k))=\Hom(\omega_R/\m\omega_R,E(\k))$. Thus $$\type(R)=\ell\Big(\frac{\Hom_R(\m\omega_R,\omega_R)}{\Hom_R(\omega_R, \omega_R)}\Big).$$ Using the identification with colons as discussed before the theorem, we get that 
    \begin{align*}
    \type(R)&=\ell\Big(\frac{\omega_R:\m\omega_R}{\omega_R:\omega_R}\Big)\\&=\ell\Big(\frac{(\omega_R:\omega_R):\m}{\omega_R:\omega_R}\Big)=\ell\Big(\frac{R:\m}{R}\Big).    
    \end{align*}    
    The first equality follows from Remark \ref{remark on HomI,J} while the last equality follows from \cite[Proposition 3.3.11(c)(ii)]{bruns_herzog_1998}. Then $\ell(R:\m/R)=|\left\{v(R:\m)\setminus v(R)\right\}|$. Since $v(R:\m)\subseteq \{x\in \mathbb{Z}\mid x+h\in v(R) \text{~for all~} h\in v(R)\}$, we have finished the proof.
\end{proof}

%The last quantity in the above statement is precisely the type of the numerical semi-group ring $S=\k\ps{v(R)}$, and the set under consideration is denoted by $\operatorname{PF}(v(R))$. 
Thus the proof above shows that the type of $R$ and the type of $\k\ps{v(R)}$ need not be the same. In fact, here is a concrete example which reflects this. 

\begin{example}\label{type may increase}\cite[Example II.1.19]{Barucci-Froberg-memoirs}
    Let $R=\k\ps{t^4,t^6+t^7,t^{11}}$. Then \\$v(R)=\{0,4,6,8,10,11,12,13,14,....\}$ (we used that $t^{10}=t^4(t^6+t^7)-t^{11}\in R$ and $t^{13}=(t^6+t^7)^2-(t^4)^3-t^{10}t^4\in R$). Also, $v(R:\m)=\{0,4,6,7,8,9,...\}$. Thus, $\type(R)=|v(R:\m)\setminus v(R)|=2$. However, $\operatorname{PF}(v(R))=\{2,7,9\}$.
\end{example}

%The study of reduced type, however, is invariant under the study of $R$ versus the study of $S=\k\ps{v(R)}$, as the following proposition shows.

\begin{proposition}\label{transitiontosemigroups} The following statements hold.
\begin{enumerate}
\item The ring $R$ is of minimal reduced type if and only if $\k\ps{v(R)}$ is of minimal reduced type.
   \item The ring $R$ is of maximal reduced type if $\k\ps{v(R)}$ is of maximal reduced type. The converse holds if additionally $\type(R)=\type(\k\ps{v(R)})$.
   \end{enumerate}
\end{proposition}

\begin{proof}
    The case of minimal reduced type (i.e., when $s(R)=1$) follows from \Cref{redtypestaysunchanged}. For (2), we use \Cref{typecomparison} and \Cref{redtypestaysunchanged}. 
\end{proof}

The converse of Statement (2) in \Cref{transitiontosemigroups} need not hold without the added assumption. The following example illustrates this.

\begin{example}
Let $R$ be as in \Cref{type may increase}. Then from the description of $v(R)$, we see that $c=10$. Thus from \Cref{redtypeinterpret}, we get that $s(R)=2$ since $\{7,9\}=[10-4,10-1]\setminus v(R)$. Thus $R$ is of maximal reduced type but $\k\ps{v(R)}$ is not since its type is $3$ whereas reduced type is $2$.
\end{example}

\begin{corollary}
    Assume $R$ is of minimal multiplicity. Then $R$ is of maximal (resp. minimal) reduced type if and only if $\k\ps{v(R)}$ is of maximal (resp. minimal) reduced type. 
\end{corollary}

\begin{proof}
It is well-known that for a Cohen-Macaulay local ring with $e(R)>1$, $e(R)\geq \type(R)+1$ and $R$ is of minimal multiplicity if and only if $e(R)=\type(R)+1$ \cite[Fact 2.6]{herzog2021tiny}. 

Since $e(R)=e(\k\ps{v(R)})$, we get that $e(\k\ps{v(R)})=\type(R)+1\leq \type(\k\ps{v(R)})+1$ using \Cref{typecomparison}. But $\type(\k\ps{v(R)})+1\leq e(\k\ps{v(R)})$ as discussed before. 

Thus, $\type(v(R))=\type(\k\ps{v(R)})$ and now \Cref{transitiontosemigroups} finishes the proof.
\end{proof}

The above discussions show that the study of maximal or minimal reduced type to complete numerical semigroup rings gives us information on the maximal or minimal reduced type for rings $R$.

\noindent \textbf{About the notation for $R$}: In the sequel, we will be studying the maximal and minimal reduced type for complete numerical semigroup rings. Essentially, as the above results suggest, studying the properties for $\k\ps{v(R)}$ gives us information on the properties for $R$. Now notice that $a_1,\dots,a_n$ may not be a generating set for $v(R)$ even though $R=\k\ps{\alpha_1t^{a_1},\dots,\alpha_nt^{a_n}}$. For example, for the ring $R=\k\ps{t^4,t^6+t^7,t^{11}}$, the complete numerical semigroup ring is $\k\ps{t^4,t^6,t^{11},t^{13}}$ as opposed to $\k\ps{t^4,t^6,t^{11}}$. Thus to avoid any confusion, we will be using a more general notation $\k\ps{H}$ for a complete numerical semigroup as opposed to $\k\ps{v(R)}$.

\begin{proposition}\label{mainpropo}
For a numerical semigroup $H$, the ring $S=\k\ps{H}$ has maximal reduced type if and only if $\min_{x\in \operatorname{PF}(H)} x\geq c-b_1$.
\end{proposition}

\begin{proof}It is clear that $[c-b_1,c-1]\setminus v(S)\subseteq \operatorname{PF(H)}$ by the definition of the latter. (This is another way of seeing in this case that $s(S)\leq \type(S)$). Thus, $S$ is of maximal reduced type if and only if $\operatorname{PF}(H)=[c-b_1,c-1]\setminus v(S)$ if and only if $\min_{x\in \operatorname{PF}(H)} x\geq c-b_1$.
\end{proof}

\begin{proposition}\label{minredtypemainprop}
For a numerical semigroup $H$, the ring $S=k\ps{H}$ is of minimal reduced type if and only if $\ds \max_{x\in \{\PF(H)\setminus \{c-1\}} x < c-b_1$.
\end{proposition}

\begin{proof}
    It is clear that $s(S)=1$ if and only if $c-1$ is the only integer in the interval $[c-b_1,c-1]$ that should be missing from $v(S)$. Hence, all the other elements in $\PF(H)$ must be strictly less than $c-b_1$. 
\end{proof}

\begin{remark}
    In fact, it is clear that if $h\in H$, then $c-1-h\not \in \PF(H)$. Hence, the above statement can also be changed to minimal reduced type if and only if $\ds \max_{x\in \{\PF(H)\setminus \{c-1\}} x \leq  c-b_1-2$. 
\end{remark}
\begin{corollary}\label{minmultcrit}
For a numerical semigroup $H$, the ring, suppose $S=\k\ps{H}$ has minimal multiplicity. Then $S$ is of maximal reduced type if and only if $b_2\geq c$.
\end{corollary}

\begin{proof}
Since $S$ has minimal multiplicity, $$\operatorname{PF}(H)=\{b_2-b_1, b_3-b_1,\dots, b_n-b_1\};$$ this follows from \cite[Proposition 2.20, Proposition 3.31]{rosales2009numerical} Thus, $S$ has maximal reduced type if and only if $b_2-b_1\geq c-b_1$ which completes the proof. 
\end{proof}
\begin{corollary}\label{minmultminredtype}
Suppose $S=\k\ps{H}$ has minimal multiplicity. Then $S$ is of minimal reduced type if and only if $b_{n-1}+b_1-1 < b_n$.    
\end{corollary}
\begin{proof}
Since $S$ has minimal multiplicity, $$\operatorname{PF}(H)=\{b_2-b_1, b_3-b_1,\dots, b_n-b_1\};$$ this follows from \cite[Proposition 2.20, Proposition 3.31]{rosales2009numerical}. Thus, $S$ has minimal reduced type if and only if $b_{n-1}-b_1 < c-b_1$. But notice that $c-1=b_n-b_1$ or $c=b_n-b_1+1$. Thus $S$ has minimal reduced type if and only if 
\begin{align*}
    b_{n-1}-b_1 < c-b_1=b_n-b_1+1-b_1=b_n-2b_1+1.
\end{align*}
\end{proof}
\begin{comment}
{\color{red} 
\begin{definition}
    Suppose the valuation semigroup $H=v(R)$ is generated by $a_1,\dots,a_n$ where $a_i=ha_1+(i-1)d, 2\leq i\leq n$ for some integer $h,d$ with $\gcd(a_1,d)=1$, then we say that $v(R)$ is generated by a generalized arithmetic sequence.
\end{definition}

\begin{proposition}
    Suppose the valuation semigroup of $R$ is generated by a generalized arithmetic sequence with $n=a_1$. Then $R$ is of minimal reduced type if and only if $d>a_1-1$. 
\end{proposition}
 
\begin{proof}
The given hypothesis implies that $R$ is of minimal multiplicity. Hence, $s(R)=1$ if and only if $a_n>a_{n-1}+a_1-1$ by \Cref{minmultminredtype}. This latter condition holds if and only if $$ha_1+(n-1)d-ha_1-(n-2)d-a_1+1>0$$ which holds if and only if $d>a_1-1$.
\end{proof}
}
\end{comment}
In general, getting a complete classification of numerical semi-group rings of maximal or minimal reduced type seems tricky. We however can describe the condition abstractly using the notion of Ap\'ery sets as well.

Recall that given $h\in H$, the Ap\'ery set of $h$ in $H$ is defined to be $$\operatorname{Ap}(H,h):=\{s\in H\mid s-h\not \in H\}.$$ One defines an order relation over the set of integers as follows: $a\leq_H b$ if $b-a\in H$. Let $M(h)=\operatorname{Maximal}_{\leq_H}\operatorname{Ap}(H,h)$. 

\begin{proposition}
Suppose $S=\k\ps{H}$ is a complete numerical semi-group ring as above. Then $S$ is of maximal reduced type if and only if $$\min_{w\in M(e(R))}w\geq (\max \operatorname{Ap}(H,e(R))-e(R)+1.$$
\end{proposition}

\begin{proof}
By \cite[Proposition 2.20]{rosales2009numerical}, we get that $\operatorname{PF}(H)=\{w-e(R)\mid w\in M(e(R))\}$. Furthermore, by \cite[Proposition 2.12]{rosales2009numerical}, we have that $c=(\max \operatorname{Ap}(H,e(R))-e(R)+1$. The proof is now complete by \Cref{mainpropo}.
\end{proof}

\subsection{Almost Gorenstein Rings}
\begin{definition}\cite[Definition-Proposition 20]{Barucci-Froberg97}
The ring $R$ is called almost Gorenstein if $$\ell(\overline{R}/R)=\ell(R/\cC)+\type(R)-1.$$
    
\end{definition} 
This class of rings has recently attracted a lot of attention. See for instance \cite{goto2013almost}, \cite{goto2015almost}, \cite{herzog2019trace}, \cite{DAO2021106655}, etc. among other sources.

The characterization of almost Gorenstein property for numerical semigroup rings due to Nari is given below.
\begin{theorem}{\cite[Theorem 2.4]{nari2013symmetries}}\label{nari theorem on almost GOrenstein}
    Suppose $H$ is a numerical semigroup  with pseudo-Frobenius set $\PF(H)=\{f_1,\dots,f_{t}\}$ and $\k\ps{H}$ is a numerical semigroup ring. Then $H$ is almost Gorenstein if and only if $f_i+f_{t-i}=c-1$ for $1\leq i\leq c-1$.
\end{theorem}
\begin{proposition}
    Assume $R$ is almost Gorenstein. Then $R$ is of maximal (resp. minimal) reduced type if and only if $k\ps{v(R)}$ is of maximal (resp. minimal) reduced type. 
\end{proposition}

\begin{proof}
    Since $R$ is almost Gorenstein, we have $\type(R)=\type(\k\ps{v(R)})$ by \cite[Proposition 29]{Barucci-Froberg97}. The proof is now complete by \Cref{transitiontosemigroups}.
\end{proof}

\begin{proposition}\label{AlmostGorensteinVal}
    Suppose the numerical semi-group ring $S=k\ps{H}$ is almost Gorenstein. Then the following statements hold. 
    \begin{enumerate}
        \item If $S$ is maximal reduced type then $c\leq 2b_1-1$,
        \item If $S$ is of minimal reduced type but not Gorenstein  then $c\geq 2b_1+3$.
    \end{enumerate}
\end{proposition}

\begin{proof}
    Let $r=\type(R)$ and $\PF(v(R))=\{f_1,\ldots, f_{r-1},c-1\}$ with $f_i<f_{i+1}$. From \Cref{nari theorem on almost GOrenstein}, we see that $R$ is almost Gorenstein if and only if $f_i+f_{r-i}=c-1$ for $1\leq i\leq r-1$. 

    For $(a)$, assume that $R$ has maximal reduced type. Hence, by \Cref{mainpropo}, $f_i\geq c-b_1$ for all $i$. Thus, $c-1=f_i+f_{r-i}\geq 2c-2b_1$ which upon simplification yields our desired conclusion.

    For $(b)$, assume $R$ is of minimal reduced type. This implies that $f_{r-1}\leq c-b_1-2$. Thus, $c-1=f_1+f_{r-1}\leq 2c-2b_1-4$. Simplifying this, we get $c\geq 2b_1+3$.
\end{proof}

The following example shows that the converse of the above statements do not hold in general.

\begin{example}
    Let $R=\k\ps{t^8,t^{11},t^{12},t^{13},t^{15},t^{17},t^{18}}$. Here $\PF(H)=\{4,5,7,9,10,14\}$ and hence $R$ is almost Gorenstein with $\type(R)=6$. Note however that $s(R)=4$. Thus, $R$ is not of maximal reduced type even though $c\leq 2b_1-1$ (here, $c=15, b_1=8$).                         
\end{example}
%{\color{red}Is there a counterexample to statement (2)?}

\begin{example}\cite[Example 3.9]{eto2017almost}
    Let $R=\k\ps{t^{34}, t^{51}, t^{53}, t^{70}}$. This is almost Gorenstein of type $3$ since $\PF(H)=\{17, 814, 831\}$. Here $c=832$ and hence satisfies $c\geq 2b_1+3$. However, notice that $s(R)=2$. 
\end{example}

However, in almost Gorenstein numerical semigroup rings of type 2, the situation is very clear as the following corollary shows. Such rings are called \textit{pseudo-symmetric} numerical semi-group rings.

\begin{corollary}\label{pseudosymm}
    Suppose $S=\k\ps{H}$ is an almost Gorenstein numerical semi-group ring of type $2$. Then the following statements hold.
    \begin{enumerate}
        \item     $S$ is maximal reduced type if and only if $c\leq 2b_1-1$,
        \item $S$ is of minimal reduced type if and only if $c>2b_1-1$.
    \end{enumerate}
\end{corollary}
\begin{proof}
    One of the directions for $(1)$ follows from \Cref{AlmostGorensteinVal}. For the converse, we first observe that since $\type(S)=2$, in the notation of the proof of \Cref{AlmostGorensteinVal}, we get that $2f_1=c-1$. Hence, $\PF(H)=\{\frac{c-1}{2},c-1\}$.
    %(in fact, this is the characterizing property of an almost symmetric numerical semi-group ring of type 2; see for example, \cite[Corollary 4.16]{rosales2009numerical}). 
    Since $c\leq 2b_1-1$, we get that $\frac{c-1}{2}\geq c-b_1$ and hence $R$ is of maximal reduced type by \Cref{mainpropo}. This proves (1).
    
    Since  $\type(S) = 2$, if the ring is not of maximal reduced type, then it must be that $s(S)=1$. Thus, (2) follows directly from (1). 
\end{proof}

The following examples illustrate both the cases.
\begin{example}\label{eg1redtypegluing}
    Let $R=\k\ps{t^{4},t^{9},t^{11}}$. Then $\PF(H)=\{7,14\}$ and hence $R$ is almost Gorenstein of type 2. Moreover, $s(R)=1$ by direct computation as well as by \Cref{pseudosymm}.
\end{example}

\begin{example}\label{eg2redtypegluing}
Let $R=\k\ps{t^{5},t^{6},t^{7},t^9}$. Here $\PF(H)=\{4,8\}$ and hence $R$ is almost Gorenstein of type $2$. Moreover, $s(R)=2$ by direct computation as well as by \Cref{pseudosymm}.
\end{example}

\subsection{Far Flung Gorenstein Rings}\label{section of far flung gorenstein}
Recently the notion of far-flung Gorenstein rings were introduced in \cite{herzog2021tiny}. These are rings where the trace ideal of the canonical module equals $\cC$. We do not go into the details of trace ideals as we will not be using the notion in this article. We refer the reader to \cite{LindoTrace}, \cite{goto2020correspondence} and \cite{herzog2019trace} for further details. These rings are (evidently) non-Gorenstein, and hence $\type(R)\geq 2$ (thus, $a_1=e(R)\geq 3$).

\begin{proposition}\label{redtype2}
Suppose $S=\k\ps{H}$ is a far-flung Gorenstein complete numerical semigroup ring. Then $s(S)\geq 2$. 
\end{proposition}

\begin{proof}
It is well-known that there is a canonical module $\ds \omega_S=\sum_{\alpha\in \operatorname{PF}(H)}St^{c-1-\alpha}$ where $c$ is the conductor valuation. Notice that $c-1\in \operatorname{PF}(H)$ always. Thus, $t^0=1$ is a generator of $\omega_S$. Moreover, from \cite[Theorem 5.1]{herzog2021tiny}, we get that $t^1$ is also a generator. Hence, $\alpha=c-2\in \operatorname{PF}(H)$ as well. But notice that $c-2\in [c-b_1,c-1]\setminus v(R)$. Hence, $s(S)\geq 2$. 
\end{proof}

\begin{theorem}\label{ffg type leq 3}
Suppose $S=\k\ps{H}$ is a far-flung Gorenstein complete numerical semi-group ring of $\type(S)\le 3$, Then $S$ has maximal reduced type.

\end{theorem}

\begin{proof}
From \Cref{redtype2}, the case when $\type(S)=2$ is clear. 

Suppose $\type(S)=3$ and $S$ has minimal multiplicity. Then we know that $b_1=e(S)=4=\edim(S)$ (\cite[Proposition 3.1]{Sally80}). As in the proof of \Cref{minmultcrit}, we get that $$\operatorname{PF}(H)=\{b_2-4,b_3-4,b_4-4\}$$ and hence $c-1=b_4-4$ or $b_4=c+3$. Using the proof of \Cref{redtype2}, we see that $c-2\in \operatorname{PF}(H)$ and hence, $b_3=c+2$. By \cite[Corollary 6.2]{herzog2021tiny}, we know that $$\{2b_4-3,\cdots, 2b_4\}\subseteq \{b_i+b_j, 2\leq i,j\leq 4\}.$$
Notice that $2b_4=2c+6, 2b_4-1=2c+5=b_3+b_4, 2b_4-2=2c+4=2b_3$. If $b_2<c$, $2b_4-3$ cannot be generated by $b_i+b_j$. Hence $b_2\geq c$ and so we are done by \Cref{minmultcrit}.

Finally, suppose $\type(S)=3$ and $S$ is not of minimal multiplicity.  First, we  notice that $b_1=e(S)=5$ by \cite[Corollary 5.3]{herzog2021tiny}. Choosing $\omega_S=\sum\limits_{\alpha\in \operatorname{PF}(H)}St^{c-1-\alpha}$ as in the proof of \Cref{redtype2}, we know that $t^0,t^1$ are generators of $\omega_S$, i.e., $c-1, c-2\in \operatorname{PF}(H)$. Since $\type(S)=3$, \cite[Proposition 6.1 (ii)]{herzog2021tiny} gives us that $\omega_S=\langle1,t,t^2\rangle$ or $\omega_S=\langle1,t,t^3\rangle$. 

Assume that $\omega_S=\langle 1,t,t^2\rangle$. Then we know that $$\operatorname{PF}(H)=\{c-3,c-2,c-1\}.$$ Since $c-3>c-5$, we are done by \Cref{mainpropo}. Next assume that $\omega_R=\langle 1,t,t^3\rangle$. Then $$\operatorname{PF}(H)=\{c-4,c-2,c-1\}$$ and again \Cref{mainpropo} finishes the proof. 
%{\color{red} DO WE NEED THIS? there are two possible choices for $\operatorname{PF}(H)$:
%\begin{align*}
%    \operatorname{PF}(H) &= \{5m+1,5m+2,5m+3\}\\
%    \operatorname{PF}(H) &=\{5m+2,5m+3,5m+4\}
%\end{align*}
%for some $m\geq 0$.  Since $c-1$ is the largest element in $\operatorname{PF}(H)$, we get that $c=5m+4$ or $c=5m+5$. In the first case, we have that $c-5=5m-1<5m+1$ whereas in the second case, we get $c-5=5m<5m+2$. Thus, \Cref{mainpropo} finishes the proof.}
\end{proof}
However, even if $S$ is far-flung Gorenstein numerical semigroup of minimal multiplicity with $\type(S)\geq 4$, it is not guaranteed that the ring will be of maximal reduced type. The following example illustrates this.

\begin{example}
    Let $R=\k\ps{t^5,t^{11},t^{17},t^{18},t^{19}}$. This ring is of minimal multiplicity with $e(R)=5, \type(R)=4, n=5$. Also notice that $R$ is far-flung Gorenstein by \cite[Corollary 6.2]{herzog2021tiny}. However, $a_2=11<c=15$ and hence $R$ is not of maximal reduced type by \Cref{minmultcrit}.
\end{example}
If the ring is not of minimal multiplicity, then far flung Gorenstein rings of type $4$ indeed have maximal reduced type.
\begin{theorem}\label{ffg type 4 non min mult}
    Suppose $S$ is a far flung Gorenstein complete numerical semi-group ring of $\type (S)=4$ which has non minimal multiplicity. Then $S$ has maximal reduced type.
\end{theorem}
\begin{proof}
    As in the previous proof, choose $\omega_S=\sum\limits_{\alpha\in \operatorname{PF}(H)}t^{c-1-\alpha}$. Since $\type(S)=4$, we have $\omega_S=\langle 1,t,t^a,t^b\rangle$ where $a<b$. Also, it follows that $6\leq e(S)\leq 9$ (\cite[Corollary 5.3]{herzog2021tiny}).  Since $S$ is far flung Gorenstein, by \cite[Proposition 6.1]{herzog2021tiny},
    \begin{align}\label{type 4 non minmal multiplicity eq1}
    \{0,1,2,3,\cdots, e(R)-1\}\subseteq \{(c-1-\alpha) + (c-1-\beta)\mid \alpha, \beta\in \operatorname{PF}(H)\}.
    \end{align}

    Thus we have two cases for the description of  $\omega_R$:
    \begin{enumerate}
        \item $\omega_S=\langle 1,t,t^2,t^b\rangle$
        \item $\omega_S=\langle 1,t,t^3,t^b\rangle$
    \end{enumerate}
    Notice that $c-1-b$ is a minimal element in $\PF(H)$, the pseudo Frobenius set. Thus $R$ is of maximal reduced type, if and only if  $c-1-b\geq c-b_1$ or $e(S)=b_1\geq b+1$.

    Suppose $e(S)\leq b$. Now if case (1) holds, then $B\cap [5,e(S)-1]=\emptyset$ which contradicts  \eqref{type 4 non minmal multiplicity eq1}. On the other hand, if case(2) holds, then $B\cap \{5\}=\emptyset$, again contradicsts \eqref{type 4 non minmal multiplicity eq1}.  Thus we have $e(S)\geq b+1$ and hence $S$ is of maximal reduced type.
\end{proof}
When the type is more than $4$, we need more tools to show that it is of maximum reduced type. In \cite[Corollary 5.3]{herzog2021tiny}, the authors connected the  the maximum bound for the multiplicity with  the \textit{Rohrbach number} $\overline{n}(r)$ where $r=\type(R)$.  The Rohrbach number is defined as follows: Suppose $A$ is a set of non negative integers with cardinality $r$. Let $n(A)$ is the integer such that  $A+A=\{a+b~|~a,b\in A\}$ contains consecutive  integers $0,1,\dots n(A)-1$ but not $n(A)$. The Rohrbach number is the integer
\begin{align*}
\overline{n}(r)=\max\{ n(A)~|~|A|=r\}.
\end{align*}
\begin{theorem}\label{ffG rohrbach maxreduced type}
    Suppose $S=\k\ps{H}$ is a far-flung Gorenstein numerical semigroup ring of type $r$. If $e(S)>\overline{n}(r-1)+1$, then $S$ has maximal reduced type.
\end{theorem}
\begin{proof}
    Since $S$ is a far-flung Gorenstein numerical semigroup ring, by \cite[Theorem 5.1]{herzog2021tiny}, there exists a canonical module $C=\langle f_1,\dots,f_r\rangle$ where $n_i=v(f_i), n_1<n_2<\cdots <n_r$ such that $\{0,1,\dots,e-1\}\subseteq A+A$ where $A=\{n_1,\dots,n_r\}$. Here $e=e(S)$.

    Suppose $n_r>e-1$. Then $\{0,1,\dots,e-1\}\subseteq A'+A'$ where $A'=\{n_1,\dots,n_{r-1}\}$. Thus the consecutive set of numbers generated by $A'+A'$ is $0,1,\dots,n(A')-1$. But by hypothesis, $\overline{n}(r-1)<e-1$ and thus, by definition we have $n(A')\leq \overline{n}(r-1)<e-1$. This contradiction forces $n_r\leq e-1$. Thus the pseudo Frobenius set $PF(H)=\{c-1-n_1,\dots,c-1-n_r \}$. Since $n_r\leq e-1$, we get $c-1-n_r\geq c-b_1$. Thus by \Cref{mainpropo}, we have $R$ is of maximum reduced type.
\end{proof}
The bound $\overline{n}(r-1)<e-1$ may not be sharp. In the following result, $r=5, \overline{n}(r)-r+1=9$.
\begin{proposition}\label{ffG rohrback type 5}
    Suppose $S$ is far flung Gorenstein complete numerical semigroup ring of  $\type (S)=5$. If $e(S)\geq 9$, then $S$ is of maximal reduced type.
\end{proposition}
\begin{proof}
        As in the previous proof, choose $\omega_S=\sum\limits_{\alpha\in \operatorname{PF}(H)}t^{c-1-\alpha}$. Since $\type(S)=4$, we have $\omega_S=\langle 1,t,t^a,t^b,t^d\rangle$ where $a<b<d$. Also, it follows that $7\leq e(S)\leq 13$ (\cite[Corollary 5.3]{herzog2021tiny}).  Since $S$ is far flung Gorenstein, by \cite[Proposition 6.1]{herzog2021tiny},
    \begin{align}\label{type 4 non minmal multiplicity eq1}
    \{0,1,2,3,\cdots, e(S)-1\}\subseteq \{(c-1-\alpha) + (c-1-\beta)\mid \alpha, \beta\in \operatorname{PF}(H)\}.
    \end{align}

    Thus we have two cases for the description of  $\omega_S$:
    \begin{enumerate}
        \item $\omega_S=\langle 1,t,t^2,t^b,t^d\rangle$
        \item $\omega_S=\langle 1,t,t^3,t^b,t^d\rangle$
    \end{enumerate}
    Notice that $c-1-d$ is a minimal element in $\PF(H)$, the pseudo Frobenius set. Thus $S$ is of maximal reduced type, if and only if  $c-1-d\geq c-b_1$ or $e(S)=b_1\geq d+1$.

    Suppose $e=e(S)\leq d$. Now suppose case (1) holds. Let $A=\{0,1,2,b,d\}$. Since $A+A$ should  generate consecutive integers $1,\dots,e-1$ and $e\leq d$, we see that $A'+A'$ where $A'=\{0,1,2,b\}$ should generate $\{0,\dots,e-1\}$. Thus our choice should satisfy $b\leq 5$. Our choices for $b$ are $3,4,5$. But with these choices, one can easily see that either $7$ or $8$ is missing from the set $A'+A'$. This is a contradiction.

    Now suppose (2) holds. Let $A=\{0,1,3,b,d\}$. Since $A+A$ should  generate consecutive integers $1,\dots,e-1$ and $e\leq d$, we see that $A'+A'$ where $A'=\{0,1,2,b\}$ should generate $\{0,\dots,e-1\}$. Thus our choice should satisfy $b\leq 5$. Our choices for $b$ are $2,4,5$. But with these choices, one can easily check that either $6$ or $7$ is missing from the set $A'+A'$.  This is a contradiction. 
    
    Thus these contradictions force $e=b_1\geq d+1$.
\end{proof}
The hypothesis of \Cref{ffG rohrbach maxreduced type} is not sharp even when $\type(S)=6$ as the following example shows:
\begin{example}\label{eg3redtypegluing}
    Let $R=\k\ps{t^{12},t^{13},t^{14},t^{15},t^{16},t^{19}}$. Then conductor $\C_R=(t^{24})$. The pseudo Frobenius set can be computed to be $\PF(H)=\{17,18,20,21,22,23\}$. Here the reduced type $s=6$ and also equals the type $r=6$. Notice that the multiplicity $e=b_1=12<n(5)=13$.
\end{example}
Pushing the proof along the same lines, one can hope to ask if such tight bounds for the multiplicity exists for the ring $S$ to have maximal reduced type.
\begin{question}
    If $S$ is far flung Gorenstein complete numerical semigroup with $e(S)\geq \overline{n}(r)-r+1$ where $r=\type(S)$, then $S$ is of maximal reduced type.
\end{question}
So far, all the examples of far-flung Gorenstein rings with non-minimal multiplicity we have tried satisfy $s(R)=\type(R)$. So it is natural to ask the following question.

\begin{question}
If $R$ is far-flung Gorenstein with $e(R)>\edim(R)$, does $R$ have maximal reduced type? 
\end{question}

Beyond $r=\type(R)=5$, one would need a lot more calculations and also knowledge about the valuation semigroup of the far flung Gorenstein ring. In fact even though examples are given to show that there exists far flung Gorenstein rings with multiplicity $e=\overline{n}(r)$, one may need to answer \cite[Question 5.5]{herzog2021tiny} in its entirety to improve \Cref{ffG rohrbach maxreduced type} and answer the above question.

%{\color{blue} What happens if $\type(R)=4$? 

%General Idea: $\omega_R=\langle1,t,t^a,t^b\rangle$ where $a<b$. Now the value $b$ in the set $\operatorname{PF}(H)$ corresponds to $\alpha=c-1-b$, and this is the minimum value in $\operatorname{PF}(H)$. So, $R$ is maximal reduced type if and only if $c-1-b\geq c-a_1\implies a_1\geq b+1\implies e(R)\geq b+1$.

%Since $\type(R)=4$, by \cite[Corollary 5.3]{herzog2021tiny}, we get that $6\le e(R) \le 9$. Since $$A:=\{0,1,2,3,\cdots, e(R)-1\}\subseteq B:=\{(c-1-\alpha) + (c-1-\beta)\mid \alpha, \beta\in \operatorname{PF}(H)\},$$ we must have the following two cases:
%\begin{enumerate}
%\item $\omega_R=\langle 1,t,t^2,t^b\rangle$
%\item $\omega_R=\langle 1,t,t^3,t^b\rangle$
%\end{enumerate}

%Suppose case 1 holds and let $b\geq  e(R)$, then $B\cap [5,e(R)-1]=\emptyset$, which is a contradiction to $A\subseteq B$. Thus, $b\leq e(R)-1$ which gives us the desired conclusion. 

%Suppose case 2 holds and let $b\geq e(R)$. Then $B\cap \{5\}=\emptyset$, again a contradiction.

%In fact, whatever the type is, the above argument shows that $\omega_R=\langle 0,1,2, ??\rangle$ or $\omega_R =\langle 0,1,3, ??\rangle$. So if $R$ is ffG and $\type(R)\geq 4$, then $s(R)\geq 3$.}

There are rings which are not far-flung Gorenstein but still of maximal reduced type. 

\begin{example}
Let $R=\k\ps{t^4,t^5,t^{11}}$. Here $\cC=(t^8)\overline{R}$. Also, observe that $\operatorname{PF}(H)=\{6,7\}=[8-4,7]\setminus v(R)$. Thus $s(R)=2=\type(R)$. However, $R$ is not far-flung Gorenstein. To see this, observe that $2(c-1)-e(R)+1=11$ cannot be written as a sum of two (not necessarily distinct) terms from $\operatorname{PF}(H)$. Thus, $R$ is not far-flung Gorenstein using \cite[Theorem 6.1 iii]{herzog2021tiny}.
\end{example}

%\subsection{Minimal reduced type or reduced type one}

 %Using \cite{rosales2009numerical}, we see that $\PF(H)=\{a_2-a_1,\dots,a_n-a_1\}$ as $R$ is of minimal multiplicity. Since $a_n-a_1=c-1$, we have $c=(h-1)a_1+(a_1-1)d+1$. Since $d>a_1-1$, we have $ha_1+(a_1-2)d<(h-1)a_1+(a_1-1)d+1$ or $a_{n-1}-a_1<c-a_1$. Thus $s(R)=1$ by \Cref{minredtypemainprop}.
%\end{proof}
%\begin{definition}{\cite[Chapter 3, Section 1]{rosales2009numerical}}

%\begin{enumerate}
%    \item A numerical semigroup is irreducible if it cannot be expressed as the %intersection of two numerical semigroups properly containing it.
%    \item A numerical semigroup $H$ is pseudosymmetric if it is irreducible and $c-1$ is even.
%\end{enumerate}
    
%\end{definition}

\subsection{Gluing of Numerical Semigroups}\label{section on gluing}
\begin{definition}\label{gluing}
     Let $H_1 = \langle a_1, a_2, \ldots , a_n\rangle$ and $H_2 = \langle b_1, b_2, \ldots , b_m\rangle$ be two
numerical semigroups. Take $y \in H_1 \setminus \{a_1, a_2, \ldots , a_n\}$ and $x \in H_2 \setminus \{b_1, b_2, \ldots , b_m\}$
such that $(x, y) = 1$. We say that
$$H = \langle xH_1, yH_2\rangle = \langle xa_1, xa_2, \ldots , xa_n, yb_1, yb_2, \ldots , yb_m\rangle$$
is a gluing of $H_1$ and $H_2$.
\end{definition}

%Throughout this section we denote, $\langle xH_1,yH_2\rangle$ to be the gluing of the two numerical semigroups via $x,y$ as in the above definition. 

The following results are well-known.

\begin{theorem}\cite[Proposition 6.6]{nari2013symmetries}\label{gluingproperties} Let $x,y,H_1, H_2$ be as in \Cref{gluing} and let $H=\langle xH_1,yH_2\rangle$ be the gluing. Let $S=\k\ps{H}$ and $S_i=\k\ps{H_i}$ be the corresponding semigroup rings with $c_S, c_{S_i}$ denoting the respective conductor valuations. Then the following statements hold.
\begin{enumerate}
  \item  $\ds \PF(H)=\{xf+yg+xy\mid f\in \PF(H_1), g\in \PF(H_2)\}.$
   \item $\ds c_S-1=x(c_{S_1}-1)+y(c_{S_2}-1)+xy.$
   \item $\ds \type(S)=\type(S_1)\type(S_2)$.
\end{enumerate}
    
\end{theorem}
It is natural to expect a similar statement for the reduced type of the glued numerical semigroup ring $S$. This is precisely the content of the next theorem. We first observe that the multiplicity of the glued numerical semigroup ring is given by $$\min\{x{a_1},y{b_1}\}.$$

\begin{theorem}\label{redtypegluing}
    Let $S, S_i$ be as in \Cref{gluingproperties}. Then $$s(S)\leq s(S_1)s(S_2).$$
\end{theorem}

\begin{proof}
    Let $\PF(H_1)=\{f_1<f_2<\cdots<f_{\type(S_1)}\}$ and $\PF(H_2)=\{g_1<g_2<\cdots<g_{\type(S_2)}\}$. Observe that $f_{\type(S_1)}=c_{S_1}-1, g_{\type(S_2)}=c_{S_2}-1$. Also, notice that the multiplicity of $S$ is $\min\{x{a_1},y{b_1}\}$.

Let $z\in [c_S-\min\{x{a_1},y{b_1}\}, c_S-1]\setminus v(S)$. By \Cref{redtypeinterpret}, we need to count all such $z$ to determine $s(S)$. Since $z\in \PF(H)$, by \Cref{gluingproperties}(1), we get that $z=xf_i+yg_j+xy$ for some $f_i\in \PF(H_1), g_j\in \PF(H_2)$. 

Suppose $g_j\not \in [c_{S_2}-b_1, c_{S_2}-1]$. Then we can write $g_j=c_{S_2}-b_1-k$ with $k\geq 2$ (since $c_{S_2}-b_1-1\not \in \PF(H_2)$. Thus, we get $$z=xf_i+yg_j+xy=xf_i+y(c_{S_2}-b_1-k)+xy.$$ Since $z\geq c_S-\min\{xa_1,yb_1\}$, using \Cref{gluingproperties}(2), we get that 
\begin{align*}
    xf_i+y(c_{S_2}-b_1-k)+xy & \geq x(c_{S_1}-1)+y(c_{S_2}-1)+xy+1-\min\{xa_1,yb_1\}
\end{align*} which upon simplification gives $$\min\{xa_1,yb_1\}-1\geq x(c_{S_1}-1-f_i)+y(b_1+k-1)\geq y(b_1+k-1)\geq y(b_1+1).$$ This gives $$yb_1-1\geq \min\{xa_1,yb_1\}-1\geq y(b_1+1)$$ but this is a contradiction since $y>a_1\geq 1$.
%Assume $xa_1\geq yb_1$. Then we get $yb_1-1\geq y(b_1+1)$ which is a contradiction since $y>a_1\geq 2$. Thus, $xa_1< yb_1$ which gives $yb_1-1>xa_1-1\geq y(b_1+1)$. This is again a contradiction. 
Hence, $g_j\in [c_{S_2}-b_1, c_{S_2}-1]$. 

A similar argument with $f_i$ shows that $f_i\in [c_{S_1}-a_1, c_{S_1}-1]$. Since this holds for any such $z$, we conclude that $s(S)\leq s(S_1)s(S_2).$
\end{proof}

\begin{corollary}\label{glue of min red type}
    Let $S, S_i$ be as in \Cref{gluingproperties}. If $S_1$ and $S_2$ are of minimal reduced type, then so is $S$. 
\end{corollary}

\begin{proof}
The proof is immediate from \Cref{redtypegluing}.
\end{proof}

However, the converse of the above statement need not hold, i.e., if $s(S)=1$, then  $S_1, S_2$ simultaneously need not have minimal reduced type. The following example illustrates this.

\begin{example}\label{eg4gluingconverseminredtype}
    Consider $S_1$ to be the ring in \Cref{eg1redtypegluing} and $S_2$ as in \Cref{eg2redtypegluing}. Let $x=10, y=13.$ Then $S=\k\ps{\langle xH_1,yH_2\rangle}=\k\ps{t^{40},t^{65},t^{78},t^{90},t^{91},t^{110},t^{117}}$. Here $s(S_1)s(S_2)=1\times 2=2$. However, $s(S)=1$: the conductor valuation $c_S=10(14)+13(8)+10\times 13+1=375$ using \Cref{gluingproperties}(2). The multiplicity of $S$ is $40$. From \Cref{gluingproperties}(1), \Cref{eg1redtypegluing,eg2redtypegluing}, we can compute the pseudo-Frobenius set:
    $$\PF(H)=\{252, 304, 322, 374\}.$$ Since $322<375-40=335$, we get that $s(R)=1$.
\end{example}
We do have the following statement in the converse direction which is reflected in \Cref{eg4gluingconverseminredtype}.
\begin{theorem}
    Let $S,S_i$ be as in \Cref{gluingproperties}. If $s(S)=1$, then either $s(S_1)=1$ or $s(S_2)=1$.
\end{theorem} 

\begin{proof}
    We will prove the statement by contradiction. Assume that $s(S_i)\geq 2$ for $i=1,2$. 
    %Thus $a_1\geq 3,b_1\geq 3$ as $R_1,R_2$ cannot be Gorenstein rings. 
    Let $f_i,g_j$ be chosen as in the proof of \Cref{redtypegluing}. 
    
    Consider $z_1=x(c_{S_1}-1)+yg_{\type(S_2)-1}+xy\in \PF(H)$ (the last part is due to \Cref{gluingproperties}(1)). Since $s(S_2)\geq 2$, we can write $g_{\type(S_2)-1}=c_{S_2}-b_1+k_1$ with $b_1-1> k_1\geq 0$. Hence, $z_1=x(c_{S_1}-1)+y(c_{S_2}-1)+xy+y(1-b_1+k_1)=c_S-1+y(1-b_1+k_1)$, where the last equality follows from \Cref{gluingproperties}(2). Since $s(S)=1$, we have that $\ds c_S-1+y(1-b_1+k_1)\leq c_S-\min\{xa_1,yb_1\}-2$ which upon simplification,  gives $\ds \min\{xa_1,yb_1\}+1\leq y(b_1-1-k_1)\leq y(b_1-1)$. Thus, we must have $xa_1<yb_1$.

    Since $s(S_1)\geq 2$, $f_{\type(S_1)-1}=c_{S_1}-a_1+k_2$ with $0\leq k_2<a_1-1$. Consider $z_2=xf_{\type(S_1)-1}+y(c_{S_2}-1)+xy=x(c_{S_1}-1)+y(c_{S_2}-1)+xy+x(1-a_1+k_2)=c_S-1+x(1-a_1+k_2)\in \PF(H)$. Since $s(S)=1$, we get that $c_S-1+x(1-a_1+k_2)\leq c_S-\min\{xa_1,yb_1\}-2$ which simplifies to $\min\{xa_1,yb_1\}+1\leq x(a_1-1-k_2)\leq x(a_1-1)$. This shows that $yb_1<xa_1$, a contradiction. 

    Thus either $s(S_1)=1$ or $s(S_2)=1$.
\end{proof}

Even if exactly one of $S_1$ or $S_2$ is of minimal reduced type, $s(S)$ need not be one. 
\begin{example}
    Let $S_1$ be the ring in \Cref{eg1redtypegluing} and let $S_2$ be the ring in \Cref{eg3redtypegluing}. Let $x=24, y=13$. Then $S=\k\ps{\langle xH_1,yH_2\rangle}=\k\ps{t^{96}, t^{156}, t^{169}, t^{182}, t^{195}, t^{208}, t^{216}, t^{264}, t^{247}}$. Here, $c_S=24(14)+13(23)+13\times 24+1=948$ whereas the multiplicity is $96$. Moreover, $\PF(H)=\{701, 714, 740, 753, 766, 779, 869, 882, 908, 921, 934, 947\}$. Since there are $6$ numbers greater than $948-96=852$, we get that $s(S)=6$. However, notice that $s(S_1)=1$ as computed in \Cref{eg1redtypegluing}.
\end{example}

Unfortunately, an analogue of \Cref{glue of min red type} for maximal reduced may not be feasible as the following example shows.

\begin{example}
    Let $S_1=\k\ps{t^{5},t^6,t^7,t^9}$ as in \Cref{eg2redtypegluing} and let $S_2=\k\ps{t^3,t^7,t^8}$. Notice that $s(S_2)=2=\type(S_2)$ since $\PF(H_2)=\{4,5\}$. Thus both $S_1$ and $S_2$ are of maximal reduced type. Choose $x=6, y=11$ to get that $S=\k\ps{\langle xH_1,yH_2\rangle}=\k\ps{t^{30},t^{33},t^{36},t^{42},t^{54},t^{77},t^{88}}$. Here, $c_S=170$ and $\PF(H)=\{134, 145, 158, 169\}$. Since $134<170-30$, we get that $s(S)=3$ whereas  $\type(S)=4$. Hence $S$ is not of maximal reduced type. 
\end{example}
However, if the glued ring has maximal reduced type, then both the starting rings must have the same as the following result shows.
\begin{theorem}
    Suppose $S,S_i$ be as in \Cref{gluingproperties}. If $S$ is of maximal reduced type, then both $S_1,S_2$ must have maximal reduced type. 
\end{theorem}

\begin{proof}
    Using \Cref{gluingproperties}(3) and the assumption that $s(S)=\type(S)$, we get that $$\type(S_1)\type(S_2)=\type(S)=s(S)\leq s(S_1)s(S_2)$$ where the last inequality follows from \Cref{redtypegluing}. Since $1\leq s(S_i)\leq \type(S_i)$ and all are integers, we get that $s(S_i)=\type(S_i)$ must hold for each $i$, thereby finishing the proof. 
\end{proof}
\subsection{Dual of numerical semigroup}\label{section on dual}

In this section, we study the behaviour of $$\End_S(\n):=\Hom_S(\n,\n)$$ where $S=\k\ps{H}$ is a complete numerical semigroup ring and $\n$ is the maximal ideal of $S$. We also make the following notation. We have a surjection
\begin{align*}
    \k\ps{Y_1,\dots,Y_m}\twoheadrightarrow \k\ps{H}
    \end{align*}
which gives us that $S=\k\ps{H}\cong\frac{\k\ps{Y_1,\dots,Y_m}}{J}$ where $J\subseteq (Y_1,\dots,Y_m)^2$. Also, $y_i$ denote the image of $Y_i$ in $S$. Thus we denote $\n=(y_1,\dots,y_m)$. Here notice that $\overline{S}=\overline{R}=\k\ps{t}$.

By identifying with colons, we see that the endomorphism ring is the same as $\ds \n:\n=S:\n$. Clearly, $H\cup \PF(H)\subseteq v(\n:\n)$. As in the proof of \Cref{typecomparison}, notice that
\begin{align*}
    \ell\left(\frac{\n:\n}{S}\right)=\type(S)=|\PF(H)|
\end{align*}
Since $\ell\left(\frac{\n:\n}{S}\right)$ computes the number of valuations $v(\n:\n)$, but not in $H$. Now the cardinality of $H\cup \PF(H)$ and $v(\n:\n)$ are the same. Thus we have the equality $H\cup \PF(H)=v(\n:\n)$.

Let $B=\n:\n$. Notice that the conductor $C_B$ of $B$ has the form $\frac{\cC}{y_1}$ where $\cC$ is the conductor of $S$. For,
\begin{align*}
    \cC_B=B:\overline{S}=(S:\n):\overline{S}=S:\n \overline{S}=S:(y_1)\overline{S}=\frac{1}{y_1}(S:\overline{S})=\frac{1}{y_1}(\cC).
\end{align*}
where $\cC=(t^c)\overline{S}$. Thus the conductor $\cC_B=(t^{c-b_1})\overline{S}$  (see for example, \cite{Barucci-Froberg-memoirs} or \cite[Proposition 3.4]{nari2013symmetries}). The above discussion is precisely the motivation of the discussion that follows.
\begin{definition}
    Let $H=\langle b_1,\dots,b_m\rangle$ be a numerical semigroup. Let $M$ denote the maximal ideal of $H$. Then the dual $H^*$ of $H$ is the numerical semigroup $M-M=\{x\in \mathbb{Z}\setminus\{0\} ~|~x+m\in H \text{ for all } m\in M\}$
\end{definition}

\begin{remark}\label{dualconstruct}
    Notice that $H\subseteq H^*, \PF(H)\subseteq H^*$. Now let $\alpha\in H^*\backslash H$, then $\alpha +m\in H$ for all $m\in M$. By definition $\alpha\in \PF(H)$. Thus $H^*=H\cup \PF(H)$. Thus for a numerical semigroup ring $\k\ps{H}$, we have $H^*=v(\n:\n)=v(B)$. Thus we can set $B=\End_S(\n)=\k\ps{H^*}$.
\end{remark}
Thus it is clear that the multiplicity can drop as we move from $H$ to $H^*$. However, the following proposition gives two conditions which guarantee that the multiplicities of $H$ and $H^*$ stay the same.

\begin{proposition}\label{dualmultiplicitydrop} The following statements hold for a numerical semigroup ring $S=\k\ps{H}$ with  $H=\langle b_1,\cdots, b_m\rangle$ and conductor valuation $c$.
     \begin{enumerate}
     \item Suppose $S$ is of minimal reduced type but not Gorenstein. Then the multiplicity of $H$ and $H^*$ are equal. Moreover, the conductor valuation $c\geq 2b_1+3$.

     \item Suppose $S$ is Gorenstein. Then multiplicity of $H^*$ is strictly less than $b_1$ if and only if $S=\k\ps{t^2,t^3}$.
     \end{enumerate}
\end{proposition}

\begin{proof}
    For (1), using \cite[Lemma 6.4]{nari2013symmetries}, we observe that $c-1-f_{t-1}\leq f_1$ where $\PF(H)=\{f_1<\cdots<f_t\}$ with $t=\type(S)$. Since $S$ is of minimal reduced type, $f_{t-1}\leq c-b_1-2$ ($c-b_1-1\not\in\PF(H)$). Such an $f_{t-1}$ exists as $t\geq 2$ by assumption. Thus $c-1-f_{t-1}\geq c-1+2+b_1-c\geq b_1+1$. Thus we have
\begin{align*}
    b_1+1\leq c-1-f_{t-1}\leq f_1
\end{align*}
 which proves that the multiplicity of $H$ and $H^*$ are equal. The last statement now follows by using the above established inequalities:  $$c\geq b_1+1+f_{t-1}+1= b_1+2+f_1\geq 2b_1+3.$$

 For $(2)$, first notice that if $R=\k\ps{t^2,t^3}$, then $\PF(H)=\{1\}$ and hence $H^{*}=\langle 1 \rangle$, in particular the multiplicity dropped strictly. Conversely, assume that the multiplicity of $H^*$ is strictly less than $b_1$. Since $R$ is Gorenstein, $\PF(H)=\{c-1\}$. Hence, by \Cref{dualconstruct}, we must have $c-1\leq b_1-1$, i.e., $c\leq b_1$. But $c\in H$ and hence $c=b_1$. Thus, $\cC=\m$ which shows that $S$ has minimal multiplicity by \Cref{multiplicityandconductor}. Moreover, $b_2=b_1+1$ since $\cC=\m$. Since $R$ is Gorenstein, \cite[Fact 2.6]{herzog2021tiny} immediately yields that $b_1=2$. So, $b_2=3$ and $\edim(S)=b_1=2$. This finishes the proof.
\end{proof}

\begin{corollary}
    Let $S=\k\ps{H}$ with $c\geq b_1+1$ and and is of minimal reduced type. Then $e(\k\ps{H^*})=e(S)=b_1$.
\end{corollary}

\begin{proof}
    If $S$ is not Gorenstein, then the result follows from \Cref{dualmultiplicitydrop}(1). So assume $S$ is Gorenstein. Since $c\geq b_1+1$, $S$ cannot be $\k\ps{t^2,t^3}$ and hence the result follows from \Cref{dualmultiplicitydrop}(2).
\end{proof}

Now a natural question is to study the behaviour of $s(S)$ under the construction of the dual semigroup. The following proposition sheds light upon this study.

\begin{theorem}
    Suppose  $S=\k\ps{H}$. If $\End_S(\n)$ is of minimal reduced type and $e(\End_S(\n))=e(S)$, then $S$ is of minimal reduced type.
\end{theorem}
\begin{proof}
%Since $H^*=H\cup \PF(H)$, the multiplicity of $H$ and $H^*$ are the same. From \cite{Barucci-Froberg-memoirs} or \cite[Proposition 3.4]{nari2013symmetries}, notice that  $F(H^*)=F(H)-a_1=c-a_1-1\in \PF(H^*)$. Suppose $c-2a_1+k\in [c-2a_1,c-a_1-2]$.  If $c-2a_1+k\in H\subseteq H^*$. If $c-2a_1+k\not\in H$, then $c-2a_1+k+m\geq c-a_1$ for all $m\in M$ and hence $c-2a_1+k\in \PF(H)\subseteq H^*$. Thus every member of $[c-2a_1,c-a_1-2]$ is in $H^*$ and hence $H^*$ is of minimal reduced type.
Using \Cref{dualconstruct}, we see that $H^*=v(\End_S(\n))$. Suppose $H^*$ is of minimal reduced type and the multiplicity of $H$ and $H^*$ are equal. From the discussion at the beginning of this subsection, $\C_B=(t^{c-b_1})\overline{S}$.  Thus  $c-b_1-1\in \PF(H^*)$. Since $s(B)=1$,  every element of $[c-2b_1,c-b_1-2]$ is in $H^*=H\cup \PF(H)$. 

Let $c-b_1+k\in[c-b_1,c-2]$. Then $c-2b_1+k\in [c-2b_1,c-b_1-2]\subseteq H^*$. Thus either $c-2b_1+k\in H$ or $c-2b_1+k\in\PF(H)$. If $c-2b_1+k\in H$, then $c-2b_1+k+b_1=c-b_1+k\in H$. If $c-2b_1+k\in \PF(H)$, then $c-2b_1+k+b_1=c-b_1+k\in H$. Thus every member of $[c-b_1,c-2]$ is a member of $H$ and hence $S$ is of minimal reduced type by \Cref{minredtypemainprop}.
\end{proof}

%The non-Gorenstein assumption cannot be removed as the following example suggests.
In general, the minimality of reduced type need not stay preserved as the following examples illustrate.
\begin{example}
    Let $R=\k\ps{t^4,t^5,t^6}$. Then $\PF(H)=\{7\}$ which shows that $R$ is Gorenstein (hence $s(R)=1$). However, $B=\k\ps{H^*}=\k\ps{t^4,t^5,t^6,t^7}$ which is of minimal multiplicity (hence $\type(B)=3$) with conductor valuation $4$. Thus $s(B)=3$ by \Cref{minmultcrit}.
\end{example}

\begin{example}
    Let $R=\k\ps{t^{40},t^{65},t^{78},t^{90},t^{91},t^{110},t^{117}}$. From \Cref{eg4gluingconverseminredtype}, we know that $s(R)=1$ and $\PF(H)=\{252, 304, 322, 374\}$ (so $R$ is not Gorenstein). We construct the dual semigroup ring $$B=\k\ps{t^{40},t^{65},t^{78},t^{90},t^{91},t^{110},t^{117},t^{252},t^{304},t^{322},t^{374}}$$ which has conductor valuation $c_B=335$. Moreover, $$\PF(H^*)=\{187, 212, 213, 226, 232, 239, 257, 264, 282, 283, 284, 296, 309, 334\}.$$
    Thus $s(B)=3$.
\end{example}
Unfortunately, an analogue of the previous theorem for maximal reduced type does not hold. The following example establishes this.
\begin{example}
    Let $R=\k\ps{t^5,t^9,t^{11},t^{12}}$. Notice that $c=14$ and $\PF(H)=\{6,7,13\}$. Here $s(R)=1$ and $\type(R)=3$, thus $R$ is not of maximal reduced type. However, $B=\k\ps{t^5,t^6,t^7,t^9}$ which is of maximal reduced type by \Cref{eg2redtypegluing}.
\end{example}

The following examples show that the multiplicity can drop even when we start with a ring of maximal reduced type. Moreover, the maximality of reduced type need not stay preserved. 

\begin{example}\cite[Example 7.14]{dao2023reflexive}
    Let $R=\k\ps{t^4,t^6,t^7,t^9}$ which has $\cC=(t^6,t^7,t^8,t^9)$ and $\PF(H)=\{2,3,5\}$. Thus, $B=\k\ps{H^*}=\k\ps{t^2,t^3}$. Thus multiplicity has dropped. Notice however that both rings have maximal reduced type.
\end{example}

\begin{example}
    Let $R=\k\ps{t^5,t^7,t^9}$. Here $s(R)=2$ with $\PF(H)=\{11,13\}$ and $R$ is of maximal reduced type. Note that $B=\k\ps{H^*}=\k\ps{t^5,t^7,t^9,t^{11},t^{13}}$ is of minimal multiplicity with conductor valuation $9$. Thus, $B$ is not of maximal reduced type by \Cref{minmultcrit}.
\end{example}

\section{The categories $\rf(R)$ and $\cm(R)$}\label{categories}
An $R$-module $M$ is said to be reflexive if the natural map $M\to \Hom_R(\Hom_R(M,R),R)$ is an isomorphism whereas $M$ is Maximal Cohen-Macaulay (MCM) if $\depth(M)=\dim(R)$. Let $\cm(R)$ denote the category of maximal Cohen-Macaulay $R$-modules and let
$\rf(R)$ denote the category of reflexive $R$-modules. We say a category is of \textit{finite
type} if it has only finitely many indecomposable objects up to isomorphism. 

Since we restrict our attention to one dimensional non-regular complete local domains containing an algebraically closed field $\k$ of characteristic $0$, we know that $\ds \{\text{free $R$- modules}\}\subseteq \rf(R)\subseteq \cm(R)$; moreover, $R$ is Gorenstein if and only if $\ds \rf(R)=\cm(R)$ (\cite[Remark 3.1]{dao2023reflexive}).  Thus it follows that if $\cm(R)$ is of finite type, then $\rf(R)$ is also of finite type. 

The study of finiteness of $\cm(R)$ and the corresponding classification question of $R$ has been of interest for quite some time.  For more details on these results along with discussions on the notions of \textit{ADE}-singularities, we refer the reader to \cite[Chapter 4 \S3, Chapters 9, 10]{leuschke2012cohen}. For the convenience of the reader we also list the possible ADE hypersurface singularities as in \cite[\S 3, Page 70,71]{leuschke2012cohen}
\begin{enumerate}
    \item $(A_n) : x^2+y^{n+1}, n\geq 1$
    \item $(D_n) : x^2y+y^{n-1}, n\geq 4$
    \item $(E_6) :  x^3+y^4$
    \item $(E_7) : x^3+xy^3$
    \item $(E_8) : x^3+y^5$
\end{enumerate}

We first discuss Gorenstein rings which are precisely the rings of both maximal and minimal reduced types. The following result is well-known in the literature and provides a complete picture in the Gorenstein setup for all dimensions.
\begin{theorem}\cite[Corollary 10.7]{leuschke2012cohen}\label{ADEresult}  
%\begin{enumerate} 
Let $( R, \m , k )$ be an excellent, Gorenstein ring containing a
field of characteristic different from $2, 3, 5$, and let $K$ be an algebraic closure
of $k$. Assume $d = \dim R \geq  1$ and that $k$ is perfect. Then $\cm(R)$ is of finite 
type if and only if there is a non-zero non-unit $f \in k \ps{ x_0 , \cdots, x_d}$ such that $\widehat{R}\cong k \ps{ x_0 , \cdots, x_d}/(f)$ and $K \ps{ x_0 , \cdots, x_d}/(f)$ is a complete $ADE$ hypersurface singularity.
 \end{theorem}
Note that in our notations, we have $d=1$, $\k=k=K$ and $\widehat{R}=R$, where $\widehat{R}$ denotes the $\m$-adic completion of $R$. Thus, due to \Cref{ADEresult}, it is enough to focus on non-Gorenstein rings. 

Here, in order to study $\cm(R)$, we need the following results. %Recall that a birational extension $S$ of $R$ is a ring extension that satisfies $R\subseteq S \subseteq \Frac(R)$.
\begin{theorem}\cite[Theorem 4.10]{leuschke2012cohen}\label{cmfinite} Let $( R, \m , k )$ be an
analytically unramified local ring of
dimension one. Then $\cm(R)$ is of finite type if and only if $e(R)\leq 3$ and $\mu(\frac{\m \overline{R}+R}{R})\leq 1$.
\end{theorem}
\begin{remark}
    So, $\cm(R)$ is of finite type for $R$ only if $e(R)\leq 3$ which in turn shows that either $R$ is of minimal reduced type or of maximal reduced type. 
\end{remark}

\begin{theorem}\cite[Propositions 6.12, 7.7]{dao2023reflexive}\label{finiterefalmostgor}
Let $(R, \m)$ be a one-dimensional Cohen-Macaulay local ring.
Let $B = \End_R(\m)$. If $\cm(B)$ is of finite type, then $\rf(R)$ is of finite type. Moreover, if $R$ is almost Gorenstein, then the converse also holds.
    
\end{theorem}

\begin{proposition}
Let $R=\k\ps{X_1,\cdots , X_n}/I$ be a one dimensional complete local domain with $k$ an algebraically closed field of characteristic $0$ with $\edim(R)\geq 2$. Let the integral closure $\overline{R}=\k\ps{t}$. If $\cm(R)$ is of finite type, then $\cm(\k\ps{H})$ is of finite type for any numerical semigroup $H$ with $R\subseteq \k\ps{H}\subseteq \overline{R}$. 
\end{proposition}

\begin{proof}
    Notice that $R\subseteq \k\ps{H}\subseteq \Frac(R)$ for any such given numerical semigroup $H$. Now the result is immediate from \cite[4.14(iv)]{leuschke2012cohen}.
\end{proof}
Notice that $H$ need not be $v(R)$ in the above proof. For example if $R=\k\ps{t^3+t^4,t^5,t^7}$, then $v(R)=\langle 3,5,7\rangle$ and hence $\k\ps{v(R)}=\k\ps{t^3,t^5,t^7}$, but $R\not\subseteq \k\ps{v(R)}$. But it is however true that $R\subseteq \k\ps{t^3,t^4,t^5}$. Thus, studying $\cm(R)$ for a numerical semigroup ring of maximal or minimal reduced type can help to some extent in understanding $\cm(R)$. The link between the category of reflexive modules however is not so clear unless we have $R$ Gorenstein to begin with. Nevertheless, we restrict our attention to $R=\k\ps{H}$ for the rest of this section.

%Now, for $R$, a numerical semigroup ring, it is a natural question to ask if any information about $\rf(R)$ can be deduced about rings of maximal or minimal reduced type. 

\begin{theorem}\label{fintyperes1}
    Let $S=\k\ps{H}$ be a numerical semigroup ring which is not Gorenstein. Then $\cm(S)$ is of finite type if and only if one of the following holds:
    \begin{enumerate}
        \item $H=\langle 3,4,5 \rangle$
        \item $H=\langle 3,5,7 \rangle$.
    \end{enumerate}
    \end{theorem}

\begin{proof}
We use the same notation as in the beginning of \Cref{section on dual}. Assume $\cm(S)$ is of finite type. Then $e(S)\leq 3$ by \Cref{cmfinite} and since $S$ is not Gorenstein, we get that $\type(S)=2$, i.e., $S$ is of minimal multiplicity. Hence $H=\langle 3, b_2, b_3 \rangle$. Now again by \Cref{cmfinite}, we must have $\mu_S(\n\overline{S}+S/S)\leq 1$. From \Cref{multiplicityandconductor}, we get that $L:=\frac{\n\overline{S}+S}{S}=\frac{\n\overline{S}}{(\n\overline{S}\cap S)}=\frac{y_1\overline{S}}{\n}$. Hence $$1\geq \mu(L)=\ell\Big(\frac{y_1\overline{S}}{\n y_1\overline{S}+\n}\Big)=\ell\Big(\frac{y_1\overline{S}}{y_1^2\overline{S}+\n}\Big).$$ Now $v(y_1\overline{S})=\{3,4,5,\rightarrow\}$ whereas $v(y_1^2\overline{S}+\n)=\{3,b_2,b_3\}\cup\{6,7,8,\rightarrow\}$. Thus the only possibilities of $b_2$ are $4,5$. Suppose $b_2=4$. Notice that the ring $\k\ps{t^3,t^4}$ has the conductor $\C=(t^6)\overline{R}$. Thus any element $t^i,i\geq 6$ can be generated by $t^3,t^4$. Now since $\edim S=3$, we must have $b_3=5$. Similarly, if  $b_2=5$, then $b_3=7$ (the conductor of $\k\ps{t^3,t^5}$ is $(t^8)\overline{S}$). This finishes one direction of the proof. 

Conversely, if $H$ is of the two forms as stated, then the same argument in the above paragraph shows that $L$ is generated by at most one element in both cases. Since $e(S)=3$, we can now apply \Cref{cmfinite}.
\end{proof}

\begin{theorem}
     If $S=\k\ps{H}$ is of minimal reduced type, then $\cm(S)$ is of finite type if and only if $S$ is a complete $ADE$ hypersurface singularity of multiplicity $2$ or $3$.
\end{theorem}

\begin{proof}
    If $S$ is non-Gorenstein with $\cm(S)$ of finite type, then $S$ is of maximal reduced type using \Cref{minmultcrit} on the two forms of $H$ we obtained in \Cref{fintyperes1}. So, if $S$ is of minimal reduced type with $\cm(S)$ of finite type, then $S$ must be Gorenstein. The conclusion for the last part of the statement now follows from \Cref{ADEresult}.
\end{proof}

\begin{theorem}
    Suppose $S=\k\ps{H}$ be a numerical semigroup ring which is not almost Gorenstein. Let $B=\End_R(\m)$. Assume that $e(S)=e(B)$. Then $\cm(B)$ is of finite type if and only if one of the following holds:
    \begin{enumerate}
        \item $H=\langle 3,7,8 \rangle$
        \item $H=\langle 3, 8, 10 \rangle$.
    \end{enumerate}
    In particular, if $S$ is non-almost Gorenstein of minimal reduced type, then $\cm(B)$ is not of finite type.
\end{theorem}

\begin{proof}
Since $S$ is not almost Gorenstein, $B$ is not Gorenstein by \cite[Theorem 5.1]{goto2013almost}. So, by \Cref{fintyperes1}, we get that $B$ is constructed either by $H^*=\langle 3,4,5\rangle $ or by $H^*=\langle 3,5,7 \rangle$. Since $e(S)=3$ and $S$ is not Gorenstein, it follows that $H=\langle 3, b_2, b_3\rangle$. Note that $\PF(H)=\{b_2-3,b_3-3\}$. Since $e(B)=e(S)$, the construction in \Cref{dualconstruct} shows that $b_2-3\geq 4$. Hence $b_2\geq 7$. If $H^*=\langle 3,4,5\rangle$, then $b_2=7$ and $b_3=8$. If $H^*=\langle 3,5,7\rangle$, then $b_2=8$ and $b_3=10$. This finishes one direction of the proof. 

Conversely, given the forms of $H$ as in the statement, it is clear that $B$ is given by either $\langle 3,4,5\rangle$ or $\langle 3,5,7\rangle$. The proof is complete by \Cref{cmfinite}.

For the last part, we use \Cref{dualmultiplicitydrop}(1) along with the observation that both the forms of $H$ given in the statement of the above corollary are of maximal reduced type by \Cref{minmultcrit}.
\end{proof}

\begin{comment}
\begin{proposition}
    Let $S=\k\ps{H}$ be of minimal reduced type. Let $B=\End_R(\m)$. If $\cm(B)$ is of finite type, then $e(R)\leq 3$. 
\end{proposition}

\begin{proof}
    If $R$ is Gorenstein, then $e(B)<e(R)$ if and only if $R=\k\ps{t^2,t^3}$ by \Cref{dualmultiplicitydrop}(2) and hence we get the desired conclusion. So assume $e(B)=e(R)$. Now \Cref{ADEresult}(1) applied to $B$ shows that $e(R)\leq 3$.   
    
    If $R$ is not Gorenstein, then $e(B)=e(R)$ by \Cref{dualmultiplicitydrop}(1) and now applying \Cref{ADEresult}(1) again finishes the proof.
\end{proof}
\begin{remark}
    In fact, the proof above shows that the only possibilities of $R$ are either $\k\ps{t^2,t^3}$ or $R=\k\ps{t^3,t^{a_2},t^{a_3}}$ with $5\leq a_2\leq a_3-2$. However, not all the latter forms 
\end{remark}
\end{comment}
\begin{theorem}\label{almostGorensteinminredtyperef}
    Let $S=\k\ps{H}$ be an almost Gorenstein ring of minimal reduced type and $\type(S)\geq 2$. Then $\rf(S)$ is of finite type if and only if  $S$ has one of the following forms:
    \begin{enumerate}
        \item $\k\ps{t^3,t^7,t^{11}}$,
        \item $\k\ps{t^3,t^8,t^{13}}$.
    \end{enumerate}
    
    %$H=\langle 3,h+3,2h+3\rangle$ where $h$ satisfies the following conditions.
    %\begin{enumerate}
     %  \item $h$ is not divisible by $3$,
     %  \item $h\geq 4$,
     %  \item $\k\ps{t^3,t^h}$ is a complete $ADE$ hypersurface singularity.
    %\end{enumerate}
\end{theorem}

\begin{proof}
    Throughout the proof, we set $B=\End_S(\n)=\k\ps{H^*}$. By \Cref{finiterefalmostgor}, we know that $\rf(S)$ is of finite type if and only if $\cm(B)$ is of finite type. 

    Assume $S $ is of the forms (1) or (2). The type of $S$ is 2 and the pseudo Frobenius set is $\{4,8 \}$ or $\{5,10 \}$ respectively. Thus $B=\k\ps{t^3,t^4}$ or $B=\k\ps{t^3,t^5}$ respectively. These rings have ADE hypersurface singularities and hence $\cm(B)$ is of finite type by \Cref{ADEresult}.

  %  Assume that $H=\langle 3,h+3,2h+3\rangle$ where $h$ satisfies the three properties.  Since $R$ is of minimal multiplicity $3$, we get that $\type(R)=2$. Moreover, $\PF(H)=\{h,2h\}$ which shows that $R$ is almost Gorenstein by \cite[Theorem 2.4]{nari2013symmetries}. Since $h\geq 4$, we get that $h\leq 2h+1-3-2=c-e(R)-2$. Hence $s(R)=1$ by \Cref{minredtypemainprop}. Finally, notice that $B=\k\ps{3,h}$. So $\cm(B)$ is of finite type by assumption and \Cref{ADEresult}(2) with $\k=K=k$ and $d=1$. This finishes one direction of the proof.

Conversely, assume that $\rf(S)$ is of finite type. So $\cm(B)$ is of finite type and hence $e(B)\leq 3$ by \Cref{ADEresult}(1). Since $s(S)=1$, we get that  $e(S)\leq 3$ by \Cref{dualmultiplicitydrop}(1). Since $\type(S)\geq 2$, it forces $e(S)=3$ and $\type(S)=2$, i.e., $S$ has minimal multiplicity $3$. Hence, $\edim(S)=3$,  Let $H=\langle 3, b_2, b_3\rangle$. We can describe $S$ completely using  \cite[Lemma 4.27]{rosales2009numerical}. We include the proof here for convenience of the reader: since $S$ is almost Gorenstein of minimal multiplicity, we have that $\PF(H)=\{b_2-3,b_3-3\}$ with $b_3-3=2(b_2-3)$. Thus, taking $b_2=h+3$, we get that $b_3=2h+3$. Further, since $b_2$ is not a multiple of $3$, $3$ does not divide $h$. Finally notice that $s(S)=1$ implies that $b_2-3\leq (b_3-3+1)-3-2$. This upon simplification gives $b_3\geq b_2+4$ which in turn implies that $h\geq 4$. Thus $H=\langle 3,h+3,2h+3\rangle$ where $h\geq 4$, $h$ is not divisible by $3$ and $\PF(H)=\{h,2h\}$. So $B=\k\ps{t^3,t^h}$. Since $\cm(B)$ is of finite type, $B$ has $ADE$ hypersurface singularities by \Cref{ADEresult}. Now the only possibilities for the defining equation of $B$ are coming from $(E_6)$ or $(E_8)$ of \cite[\S 3, Page 70,71]{leuschke2012cohen}. Thus $B=\k\ps{t^3,t^4},(h=4)$ or $B=\k\ps{t^3,t^5},(h=5)$. Thus we have $R=\k\ps{t^3,t^7,t^{11}}$ or $R=\k\ps{t^3,t^8,t^{13}}$.
\end{proof}

\begin{question}
    Can we classify all numerical semigroup rings of minimal reduced type which are not almost Gorenstein but $\rf(S)$ is of finite type?
\end{question}

\begin{proposition}\label{minmult3almostgorensteinref}
    Let $S=\k\ps{H}$ be an almost Gorenstein ring of maximal reduced type $2$ and $e(S)=3$. Then $\cm(S)$ (and hence $\rf(S)$) is of finite type. 
\end{proposition}

\begin{proof}
Since $s(S)=2=\type(R)$ and $e(S)=3$, we get that $S$ is of minimal multiplicity. Since $S$ is almost Gorenstein, the same argument as in the second part of \Cref{almostGorensteinminredtyperef} shows that $H=\langle 3, h+3, 2h+3 \rangle$ where $3$ does not divide $h$. Again, $\PF(H)=\{h,2h\}$ and since $S$ is of maximal reduced type, by \Cref{minmultcrit} we get that $h+3\geq 2h+1$, i.e., $h\leq 2$. Hence, $H=\langle 3,4,5 \rangle$ or $H=\langle 3,5,7\rangle $. So $\cm(S)$ is of finite type in both cases by \Cref{fintyperes1} and consequently, $\rf(S)$ is of finite type.
%    Assume $R$ is of one of the forms given. Then $\cm(R)$ is of finite type by \Cref{fintyperes1}. Hence $\rf(R)$ is of finite type. Moreover, the respective pseudo-Frobenius sets are $\{1,2\}$ and $\{2,4\}$. Thus both $(1)$ and $(2)$ are almost Gorenstein rings of maximal reduced type $2$ by \cite[Theorem 2.4]{nari2013symmetries} and \Cref{minmultcrit}.
\end{proof}

\begin{theorem}\label{ref finite e leq 2}
     Let $S$ be a non-Gorenstein almost Gorenstein ring of maximal reduced type. Let $B=\End_S(\n)$ and assume that $e(B)\leq 2$. Then $\rf(S)$ is of finite type if and only if $S$ has one of the following forms.
     \begin{enumerate}
         \item $\k\ps{t^{b_1},t^{b_1+1},\ldots, t^{2b_1-1}}$ with $b_1 \geq 3$,
         \item $\k\ps{t^{b_1},t^{b_1+2},t^{b_1+3},\ldots,t^{2b_1-1},t^{2b_1+1}}$ with $b_1\geq 3$,
         %\item $\k\ps{t^{b_1},t^{b_1+3},t^{b_1+4},\ldots, t^{2b_1-1},t^{2b_1+1},t^{2b_1+2}}$ with $b_1\geq 4$.
     \end{enumerate}
\end{theorem}

\begin{proof}
By \Cref{almostGorensteinminredtyperef,cmfinite}, $e(B)\leq 3$ is necessary. Now the assumption on $e(B)$ implies that $S$ is of minimal multiplicity by \cite[Theorem 5.1]{goto2013almost}. So, if $S=\k\ps{t^{b_1},t^{b_2},\ldots,t^{b_n}}$, then $b_2\geq c$ by \Cref{minmultcrit}. 

It is easy to check that for the given forms of $S$, $\cm(B)$ is of finite type and hence $\rf(S)$ is of finite type. The fact that the given forms are almost Gorenstein of maximal reduced type follows from \cite[Theorem 2.4]{nari2013symmetries} and \Cref{minmultcrit}.

We discuss the converse in separate cases. Since $S$ is non-Gorenstein, $b_1=e(S)\geq 3$.
\begin{enumerate}
        \item Assume $e(B)=1$. So $B=\k\ps{t}$. Since $B=\n:\n=S:\n$, then $B=\overline{S}$ if and only if $S:\n=\overline{S}$ if and only if $S:(S:\n)=S:\overline{S}=\cC$ if and only if $\n=\cC$ where the last part follows from \cite[Proposition 2.7 and Corollary 3.2]{dao2023reflexive}. Thus, $S=\k\ps{t^{b_1},t^{b_1+1},\ldots, t^{2b_1-1}}$.

        \item Assume $e(B)=2$. Since $\PF(H)=\{b_2-b_1,\ldots, b_n-b_1\}$, \Cref{dualconstruct} shows that $b_2=b_1+2=c$. Thus, $S=\k\ps{t^{b_1},t^{b_1+2},t^{b_1+3},\ldots,t^{2b_1-1},t^{2b_1+1}}$.    
\end{enumerate}
\end{proof}

\begin{theorem}\label{ref finite e =3}
Let $S$ be a non-Gorenstein almost Gorenstein ring of maximal reduced type. Let $B=\End_S(\n)$ and assume that $e(B)=3$.
Then $\rf(S)$ is of finite type if and only if $S=\k\ps{t^{b_1},t^{b_1+1},t^{b_1+3},t^{b_1+4},\cdots,t^{2b_1-1}}, b_1\geq 4$

%then the following hold: 
%\begin{enumerate}
%    \item $B=\k\ps{t^3,t^4,t^5}$ or $B=\k\ps{t^3,t^5,t^7}$
%    \item $e(S)\geq 4$
%    \item $S$ is not of minimal multiplicity.
%    \item the conductor valuation $c\leq 2e(S)-1$.
%\end{enumerate}
    %Assume that $\rf(S)$ is of finite type. 
\end{theorem}
\begin{proof}
  We know $\cm(B)$ is of finite type by \Cref{finiterefalmostgor}.  
%\begin{enumerate}
%\item  
Assume $\edim(B)=2$, i.e., $B$ is Gorenstein  but not of minimal multiplicity. Also,  $S$ is of minimal multiplicity by \cite[Theorem 5.4]{goto2013almost}.
        %Since $\edim(B)<3$, $\type(B)=1$. Then as before, $S$ must have minimal multiplicity and hence $b_2\geq c$. Moreover, 
        Since $\cm(B)$ is of finite type, the only possibilities of $B$ are $\k\ps{t^3,t^4}$ and $\k\ps{t^3,t^{5}}$ as in the argument of \Cref{almostGorensteinminredtyperef}. If $e(S)=3$, then the proof of \Cref{minmult3almostgorensteinref} shows that $H=\langle 3,4,5 \rangle $ or $\langle 3,5,7\rangle $. But for either of these rings, $e(B)=2$, a contradiction. Hence $e(S)\geq 4$. Since $S$ is of maximal reduced type, $b_2\geq c$. Also, $\PF(H)=\{b_2-b_1,\ldots, b_n-b_1\}$. Since $3\in v(B)$, we must have $b_2=b_1+3$. If $c<b_2$, then $c=b_1$ and hence $b_2=b_1+1$, contradicting our implication that $b_2=b_1+3$. Thus $c=b_2$. Using the fact that $S$ is of minimal multiplicity, we see that $S=\k\ps{t^{b_1},t^{b_1+3},t^{b_1+4},\ldots, t^{2b_1-1},t^{2b_1+1},t^{2b_1+2}}$ with $b_1\geq 4$. However, $\PF(H)=\{3,4,\ldots,b_1-1,b_1+1,b_1+2\}$. This is not almost Gorenstein since $3+b_1+1\neq b_1+2$ by \Cref{nari theorem on almost GOrenstein}.
        
       % Hence if $S$ is almost Gorenstein of maximal reduced type, $\rf(S)$ is of finite type and $e(B)=3$, then $e(S)\geq 4$ and $B$ must be of the form $\langle 3,4,5\rangle$ or $\langle 3,5,7 \rangle$.

    %\item $S=\k\ps{t^{b_1},t^{b_1+3},t^{b_1+4},\ldots, t^{2b_1-1},t^{2b_1+1},t^{2b_1+2}}$ with $b_1\geq 4$.

%\item 
Thus  $e(B)=3=\edim(B)$, i.e., $B$ is of minimal multiplicity. 
         Then by \Cref{fintyperes1}, $B=\k\ps{t^3,t^4,t^5}$ or $B=\k\ps{t^3,t^5,t^7}$. Moreover, $S$ cannot be of minimal multiplicity and hence $b_1=e(S)\geq 4$ (because $\type(S)\geq 2$). Since $t^3\in B$, $3\in \PF(H)$. Thus $b_i=b_1+3$ for some $i\geq 2$.  Since $b_1\geq 4$ and $3\in\PF(H)$, we have $3\geq c-b_1$ as $S$ has maximal reduced type. If $3>c-b_1$, then $c-b_1$ is in the valuation semigroup of $S$ which is a contradiction as $1\leq c-b_1<3<b_1$. Notice that the first inequality is because of the fact that $S$ is not of minimal multiplicity and hence $c>b_1$. Thus $c-b_1=3$ or $c=b_1+3$ and hence $c-1=b_1+2$ is not a member of the valuation semigroup of $S$. Thus the only choice of $S$ with these properties is $S=\k\ps{t^{b_1},t^{b_1+1},t^{b_1+3},t^{b_1+4},\ldots,t^{2b_1-1}}, b_1\geq 4$. Moreover, notice that  $\PF(H)=\{3,4,5, \ldots, b_1-1,b_1+2\}$. Hence $S$ is almost Gorenstein by  \Cref{nari theorem on almost GOrenstein}. In other words $v(B)=\langle 3,4,5\rangle$ or $B=\k\ps{t^3,t^4,t^5}$. Interestingly, notice that $B$ cannot take the form $\k\ps{t^3,t^5,t^7}$ as mentioned in the beginning of this paragraph. 

 Conversely, if $S=\k\ps{t^{b_1},t^{b_1+1},t^{b_1+3},t^{b_1+4},\ldots,t^{2b_1-1}}, b_1\geq 4$, then $v(B)=\langle 3,4,5\rangle$ and hence $B=\k\ps{t^3,t^4,t^5}$. Thus $\cm (B)$ is of finite type (by \Cref{fintyperes1}), equivalently $\rf (S)$ is of finite type.
\end{proof}    
        %\end{enumerate}

\begin{question}
    Can we classify all numerical semigroup rings of maximal reduced type which are not almost Gorenstein but $\rf(S)$ is of finite type?
\end{question}

%The following examples show that the maximal reduced type case can have both $\rf(R)$ of finite and infinite types when $e(R)\geq 4$.

%\begin{example}
%    Let $R=\k\ps{t^5,t^6,t^7,t^9}$ which is almost Gorenstein of maximal reduced type $2$. Then $B=\End_R(\n)=\k\ps{t^4,t^5,t^6,t^7}$ which has $e(B)\geq 4$ and hence $\cm(B)$ is not of finite type. Using \Cref{finiterefalmostgor}, we get that $\rf(R)$ is not of finite type.
%\end{example}

%\begin{example}
%    Let $R=\k\ps{t^4,t^5,t^7}$. Here $\PF(H)=\{3,6\}$ and hence $R$ is almost Gorenstein of maximal reduced type. So $B=\k\ps{t^3,t^4,t^5}$ and thus $\cm(B)$ is of finite type by \Cref{fintyperes1}. So $\rf(R)$ is of finite type by \Cref{finiterefalmostgor}.
%\end{example}
    
%\begin{example}
%   Let $R=\k\ps{t^3,t^4,t^5}$. Then $\cC=\m$ and hence $R$ is of maximal reduced type by \Cref{lengthconditions}(1). Also, $\rf(R)$ is of finite type by \cite[Corollary 6.13]{dao2023reflexive}.
%\end{example}

%\section{Introduction}
\bibliographystyle{alpha}
\bibliography{references}
\end{document}